\documentclass[12pt]{amsart}
\usepackage{amsmath,amsfonts,amsthm,amssymb,mathrsfs}
\usepackage{epsfig}
\usepackage{slashed}
\usepackage{mathtools}
\usepackage{enumerate}
\usepackage{fullpage}
 \usepackage{soul}
\usepackage{xcolor}
\usepackage{float}

\newcommand{\mb}{\mathbf}
\newcommand{\mc}{\mathcal}

\renewcommand{\Re}{\mathrm{Re}\,}
\renewcommand{\Im}{\mathrm{Im}\,}

\newcommand{\N}{\mathbb{N}}
\newcommand{\R}{\mathbb{R}}
\newcommand{\C}{\mathbb{C}}

\newcommand{\B}{\mathbb{B}}

\DeclareMathOperator{\supp}{supp}
\DeclareMathOperator{\range}{rg}

\newtheorem{lemma}{Lemma}[section]
\newtheorem{theorem}[lemma]{Theorem}

\newtheorem{proposition}[lemma]{Proposition}
\theoremstyle{remark}
\newtheorem{remark}[lemma]{Remark}

\theoremstyle{definition}
\newtheorem{definition}[lemma]{Definition}

\numberwithin{equation}{section}

\title[]{Co-Dimension One Stable Blowup for the Quadratic Wave Equation Beyond the light cone}

\author{Po-Ning Chen}
\address{University of California, Riverside, 900 University Ave, Riverside, CA 92521}
\email{poningc@ucr.edu}

\author{Roland Donninger}
\address{Universit\"at Wien, Fakult\"at f\"ur Mathematik, Oskar-Morgenstern-Platz 1, 1090 Vienna, Austria}
\email{roland.donninger@univie.ac.at}

\author{Irfan Glogi\'c}
\address{Universit\"at Wien, Fakult\"at f\"ur Mathematik, Oskar-Morgenstern-Platz 1, 1090 Vienna, Austria}
\email{irfan.glogic@univie.ac.at}

\author{Michael McNulty}
\address{Michigan State University, 619 Red Cedar Road, East Lansing, MI 48824}
\email{mcnult50@msu.edu}

\author{Birgit Sch\"orkhuber}
\address{Universit\"at Innsbruck, Institut f\"ur Mathematik, Technikerstrasse 13, 6020 Innsbruck, Austria}
\email{Birgit.Schoerkhuber@uibk.ac.at}
\thanks{P.C. is supported by the Simons Foundation collaboration award \#584785. R.D. acknowledges support by the Austrian Science Fund FWF via Projects P 30076 and P 34560. I.G. acknowledges support by the Austrian Science Fund FWF, Projects P 30076 and P 34378.}

\begin{document}
\begin{abstract}
We study the stability of an explicitly known, non-trivial self-similar blowup solution of the quadratic wave equation in the lowest energy supercritical dimension $d=7$. This solution blows up at a single point and extends naturally away from the singularity. By using hyperboloidal similarity coordinates, we prove the conditional nonlinear asymptotic stability of this solution under small, compactly supported radial perturbations in a region of spacetime which can be made arbitrarily close to the Cauchy horizon of the singularity. To achieve this, we rigorously solve the underlying spectral problem and show that the solution has exactly one genuine instability. The unstable nature of the solution requires a careful construction of suitably adjusted initial data at $t= 0$, which, when propagated to a family of spacelike hypersurfaces of constant hyperboloidal time, takes the required form to guarantee convergence. By this, we introduce a new canonical method to investigate unstable self-similar solutions for nonlinear wave equations within the framework of hyperboloidal similarity coordinates. 
\end{abstract}

\maketitle

	\section{Introduction}
		This paper concerns the radial quadratic wave equation
			\begin{equation}
						\Big(\partial_t^2-\partial_r^2-\frac{d-1}{r}\partial_r\Big)u(t,r)=u(t,r)^2  \label{qwe}
			\end{equation}
for $(t,r)\in I\times[0,\infty)$, $I\subset\mathbb R$ an interval containing zero, and $r=|x|$ for $x\in\mathbb R^d$. Equation \eqref{qwe} exhibits the scaling symmetry $u\mapsto u_\lambda$,
			$$
				u_\lambda(t,r):=\lambda^{-2}u(t/\lambda,r/\lambda)
			$$
		for any $\lambda>0$. This rescaling leaves invariant the energy norm $\dot{H}^1(\mathbb R^d)\times L^2(\mathbb R^d)$ precisely when $d=6$ which defines the energy critical case.  Equation ~\eqref{qwe} exhibits finite-time blowup in all space dimensions as is obvious from the existence of the ODE blowup solution
 \begin{equation}\label{Def:ODEblowup}
	u_T^{ODE}(t,r) := \frac{6}{(T-t)^2}, \quad T>0,
  \end{equation} 
which is known to be stable under small perturbations locally in backward light cones, see
 \cite{DS2017}, \cite{CGS21}. Remarkably, there is an explicit non-trivial radial self-similar solution that exists in all supercritical dimensions $d\geq7$ and is given by
			\begin{equation}
				u_T^*(t,r):=\frac{1}{(T-t)^2}U\Big(\frac{r}{T-t}\Big),\quad U(\rho)=\frac{c_1 - c_2 \rho^2}{(c_3 + \rho^2)^2}, \label{ss soln}
			\end{equation}
for $T >0$ 	with
						\begin{align*}
c_1=\frac{4}{25}((3d-8)d_0+8d^2-56d+48), \quad c_2=\frac{4}{5} d_0,  \quad c_3=\frac{1}{15}(3d-18+d_0),
\end{align*}
and $d_0 =\sqrt{6(d-1)(d-6)}$.
This solution was recently introduced in \cite{CGS21} by Csobo, Glogi\'c, and Sch\"orkhuber, who established in $d=9$ its \textit{conditional asymptotic stability} without symmetry assumptions  locally in backward light cones. Rearranging the right hand side of Equation~\eqref{ss soln} yields
\begin{align}\label{Sol_contin}
	u_T^*(t,r)=\frac{c_1(T-t)^2-c_2 r^2}{\big(r^2+c_3(T-t)^2\big)^2},
\end{align}
from which it is evident that $u^{*}_T$ is well-defined for all $(t,r) \in \R \times [0,\infty)$, $(t,r)  \neq (T,0)$. Hence, in contrast to the ODE blowup and localized versions of it,
$u_T^*$ extends naturally past the blowup time and converges to zero for $t \to \infty$ in a self-similar manner. 

As the result of \cite{CGS21} is local in nature, it leaves completely open the evolution of perturbations of $u_T^*$ outside of the backward light cone and in particular, past the blowup time. We address these questions in the lowest supercritical dimension, $d=7$, where the profile in Equation~\eqref{ss soln} is given by 
 \[ U(\rho) = \frac{24(21 - 5 \rho^2)}{(3+5 \rho^2)^2}. \]
Furthermore, we restrict ourselves to the radial case. The key ingredient allowing us to access a larger region of spacetime is the coordinate system called \textit{hyperboloidal similarity coordinates}. These were first introduced by Biernat, Donninger and Sch\"orkhuber in \cite{BDS19} for the investigation of stable blowup in wave maps outside of backward light cones in $d=3$. Recently, the framework has been generalized to higher odd space dimensions by Donninger and Ostermann in \cite{DO21}. Hyperboloidal similarity coordinates are well-adapted to self-similarity much like standard similarity coordinates typically used in the study of self-similar blowup. However, they have the significant advantage that they cover regions of spacetime past the blowup time. This property is precisely what we utilize to access this larger region of spacetime.

				\subsection{Hyperboloidal Similarity Coordinates}
					In this paper, we consider \textit{radial, hyperboloidal similarity coordinates}. Namely, given $T>0$, we define the map
				\begin{align*}
					\eta_T:\mathbb R\times[0,\infty)&\to\mathbb R\times[0,\infty)
					\\
					(s,y)&\mapsto\big(T+e^{-s}h(y),e^{-s}y\big)
				\end{align*}
			where
				$$
					h(y)=\sqrt{2+y^2}-2
				$$
is referred to as the \textit{height function}. We remark that $\eta_T$ defines a diffeomorphism onto its image. The specific form of $h$ is arbitrary except for the fact that the level sets
				$$
					\{(s,y)\in\mathbb R\times[0,\infty):s=c\},\;c\in\mathbb R,
				$$
			are Cauchy surfaces which asymptote to forward light cones. Note that for the above choice of height function, $y=\frac{1}{2}$ corresponds to backward light cones. It is precisely due to the nontrivial nature of the height function that the coordinates cover a region of spacetime outside of the backward light cone of $(T,0)$, see Figure \ref{HSC_Spacetime_Diagram}. Observe that taking $h(y)=-1$ returns standard similarity coordinates.
			
				\begin{figure}
					\includegraphics[scale=0.8]{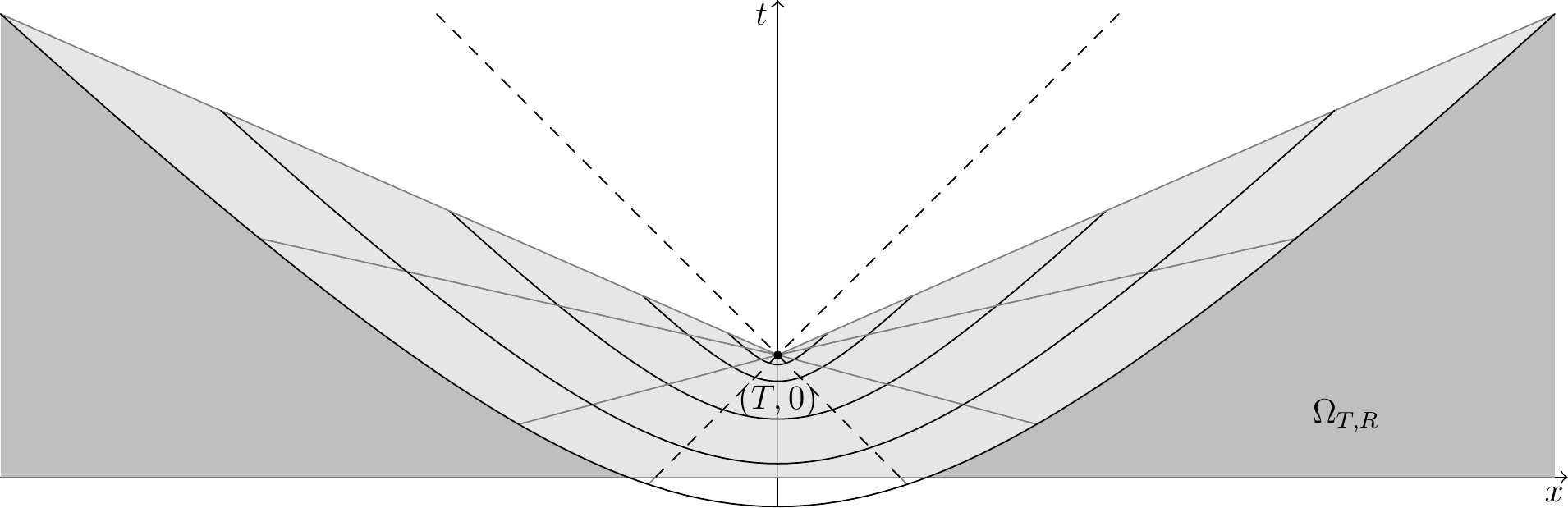}
					\centering
					\caption{A spacetime diagram depicting hyperboloidal similarity coordinates and the region $\Omega_{T,R}$. The dashed lines are the boundary of the forward and backward light cones of the spacetime point $(T,0)$. The hyperboloids correspond to level sets of $s$ and the rays emanating from $(T,0)$ correspond to level sets of $y$. The shaded region depicts the spacetime region $\Omega_{T,R}$. As $R$ increases, the upper portion of the boundary approaches the forward light cone of $(T,0)$. Within the light grey region we have convergence of the solution $u$ to $u_T^*$ according to Equation \eqref{convergence}.}
					\label{HSC_Spacetime_Diagram}
				\end{figure}
			
				\subsection{Statement of the Main Result} \label{Statement of the Main Result}
We are now ready to state our main result, which establishes the conditional asymptotic stability of $u_T^*$ in hyperboloidal similarity coordinates in $d=7$. More precisely, we prove that for every smooth, small, and compactly supported perturbation $(f,g)$ of the blowup initial data $u_1^*[0] = (u_1^*(0,\cdot), \partial_0 u_1^*(0,\cdot))$, there is a correction in terms  of fixed functions $(Y_1,Y_2)$, such that the corresponding solution $u$ blows up at $(T,0)$ for some $T >0$, exists as a smooth function in the complement of the forward light cone of $(T,0)$, and converges to $u_T^*$ on hyperboloidal time-slices as $s \to \infty$. For the precise statement, we define for $R \geq \frac12$ and $T > 0$,
\begin{align}\label{Def:OmegaTR}
							\Omega_{T,R}:=\{(t,x)\in\mathbb R\times \R^7 :0\leq t<T+ \tfrac{h(R)}{R}|x| \},
\end{align}
see Figure \ref{HSC_Spacetime_Diagram}.	We note that for large values of $R$, this region extends arbitrarily close to the forward light cone of $(T,0)$.

\begin{theorem} \label{stability on whole space}
Let $d=7$ and fix $R\geq\frac{1}{2}$. There exist positive constants $\delta,r_0,M_0,\omega_0$ and a fixed pair of radial functions $\big(Y_1,Y_2\big)\in C^{\infty}(\R^7)^2$ such that the following holds. For any pair of radial functions $(f,g)\in C^{\infty}(\R^7)^2$ supported on $\B^7_{r_0}$ and satisfying
								$$
									\|(f,g)\|_{H^{10}(\mathbb R^7)\times H^9(\mathbb R^7)}\leq\frac{\delta}{M_0^2},
								$$
							there exists $\alpha\in[-\frac{\delta}{M_0},\frac{\delta}{M_0}]$, $T\in[1-\frac{\delta}{M_0},1+\frac{\delta}{M_0}]$, and a unique solution $u\in C^\infty(\Omega_{T,R})$ to 
								\begin{equation}
								\begin{cases}
									\big(\partial_t^2-\Delta_x \big)u(t,x)=u(t,x)^2&(t,x)\in\Omega_{T,R},
									\\
									u(0,x)=u_1^*(0,|x|)+f(x)+\alpha Y_1(x) & x \in \R^7,
									\\
									\partial_0u(0,x)=\partial_0u_1^*(0,|x|)+g(x)+\alpha Y_2(x) &x\in  \R^7.
								\end{cases} \label{qwe adjusted CP}
								\end{equation}
Moreover, $u$ is radial, it blows up at $(T,0)$ and converges to $u_T^*$ in the sense that 
								\begin{align}
								\begin{split}
									e^{-2s}\|(\tilde u - u_T^*) \circ\eta_T(s,|\cdot|)\|_{H^6(\mathbb B^7_R)}\leq\delta e^{-\omega_0s}
									\\
									e^{-2s}\|\partial_s(\tilde u - u_T^*) \circ\eta_T(s,|\cdot|)\|_{H^5(\mathbb B^7_R)}\leq\delta e^{-\omega_0s} \label{convergence}
								\end{split}
								\end{align}
	for all $s\geq0$, where $u= \tilde u(|\cdot|)$.

				\end{theorem}
				
Equation \eqref{convergence} implies that the convergence takes place on shrinking hyperboloidal time-slices. These hyperboloids foliate the compact region depicted in Figure \ref{HSC_Spacetime_Diagram}. We note that
\begin{align*}
(u_T^*\circ\eta_T)(s,y)=\frac{24e^{2s}\big(21h(y)^2 - 5y^2\big)}{\big(3h(y)^2+5y^2\big)^2}.
\end{align*} 
Hence, the normalizing factor in Equation \eqref{convergence} appears naturally and reflects the behavior of the blowup solution itself. Our techniques necessitate the use of various Sobolev embeddings and this imposes the high degree of regularity appearing in Theorem \ref{stability on whole space}.

To put our result into perspective, we note that compared to
 \cite{BDS19}, \cite{DO21} which prove \textit{stable} blowup in hyperboloidal similarity coordinates for related equations, see Section \ref{Sec:FurhterResults}, the unstable nature of $u_T^*$ introduces new difficulties.  First, in order to  quantify the degree of instability of $u_T^*$, we rigorously analyze the underlying spectral problem in hyperboloidal similarity coordinates. In fact, we prove that apart from an instability induced by time-translation invariance, there is exactly one \textit{genuine} unstable direction when perturbing around $u_T^*$. The second problem concerns the design of a suitable correction of the initial data at $t = 0$, namely the pair of radial functions $(Y_1,Y_2)$ in Theorem \ref{stability on whole space},  which is able to account for this instability when being propagated to an initial hyperboloidal time-slice. These two aspects comprise the main novelties of the paper and will be elaborated on in more detail in the following remarks.

\begin{remark}\label{Rem::Codimension}
By adapting the techniques of \cite{CGS21}, one can prove a genuine co-dimension two stability result. One of these co-dimensions is due to time translation invariance of the equation while the other has no connection to any symmetry. In this way, it is possible to prove a \textit{co-dimension one, modulo symmetries} stability result. However, this requires one to restrict themselves to evolve data specified on a fixed hyperboloid in spacetime. We aim to instead prove a conditional stability result starting with data specified at $t=0$. The condition for stability is that an arbitrary perturbation needs to be suitably adjusted along a one-dimensional subspace spanned by $(Y_1,Y_2)$.
\end{remark}

\begin{remark}\label{Rem:Spec}
\textit{The spectral problem and the extension of the result to higher space dimensions.}
The spectral problem underlying the stability analysis of non-trivial self-similar solutions is very delicate, particularly in the presence of genuine instabilities. In \cite{CGS21}, perturbations of $u^*_T$ were analyzed in $d=9$ using \textit{local self-similar coordinates}, which cover the backward light cone of the blowup point. The choice of the space dimension was by no means arbitrary; by exploiting the conformal invariance of the linearized equation, a genuine unstable eigenvalue, namely $\lambda^* = 3$, and its eigenfunction could be computed \textit{explicitly} from the time-translation symmetry mode,
see also \cite{GMS20}.  This allowed to rigorously exclude the existence of other unstable spectral points. In other  supercritical space dimensions, the situation is different and a similar approach does not work. In fact, our numerics indicate that $\lambda^*$ is real, larger than two and non-integer in general, except for $d=7$, where it appears that $\lambda ^*=4$. This is exactly what we verify in the present paper.  Moreover, even without an obvious underlying symmetry, we are able to find a closed form expression for the corresponding eigenfunction, see Remark \ref{Rem:Y} below. This allows us to extend the methods of  \cite{CGS21}, \cite{GS21} to rigorously solve the spectral problem in hyperboloidal similarity coordinates. We emphasize that the spectral analysis is the key difficulty in the extension of Theorem \ref{stability on whole space} to higher space dimensions. With the results of \cite{CGS21} it is straightforward for $d=9$.
\end{remark}

\begin{remark}\label{Rem:Y}
\textit{On the adjustment terms $(Y_1,Y_2)$.}
We first observe that the linearized equation, 
			\begin{align}\label{Eq:QWE_Linearized}
				\Big(\partial_t^2-\partial_r^2-\frac{d-1}{r}\partial_r\Big)u(t,r)=2u_T^*(t,r)u(t,r),
			\end{align}
		has the following two solutions given explicitly by
			\begin{align}\label{Eq:Symm_Mode}
				F_1^*(t,r)=\frac{(T-t)\big(7(T-t)^2-15r^2\big)}{\big(5r^2+3(T-t)^2\big)^3}
			\end{align}
		and
		\begin{align}\label{Eq:Unst_Mode}
				F_4^*(t,r)=\frac{1}{\big(5r^2+3(T-t)^2\big)^3}
		\end{align}
		for $(t,r)\neq(T,0)$. After we reformulate the problem in hyperboloidal similarity coordinates, we will see that these two solutions correspond to exponentially growing solutions of the linearized equation. By rigorously analyzing the underlying spectral problem, we prove that these are the only instabilities we have to consider, see Remark \ref{Rem:Spec} above. In the nonlinear evolution, they can be accounted for in a systematic way. First, notice that the existence of the first solution $F_1^*$ is precisely due to the time translation symmetry of Equation \eqref{qwe}. More precisely, observe that $\partial_Tu_T^*=-432F_1^*$. Consequently, we can account for this instability by suitably adjusting the blowup time. On the other hand, $F_4^*$ does not appear to have a connection with any spacetime symmetry and is a genuine instability of the blowup profile. Thus, one might expect that adjusting perturbations of $u_T^*$ by some multiple of $F_4^*$ could stabilize the evolution. Unfortunately, our techniques rely crucially on the perturbation having compact support and $F_4^*$ does not have this property. Multiplying by a cutoff could fix this issue, however, doing so carelessly poses major issues in controlling the evolution of data along hyperboloids. Hence, one of the main novelties of our work lies in the construction of a suitable correction in terms of functions $(Y_1,Y_2)$ which stabilize the evolution of data close to $u_T^*$. This construction requires a delicate interplay between the standard Cauchy evolution of data along hypersurfaces of constant physical time and that of data along hypersurfaces of constant hyperboloidal time. Before carrying out the construction, we will motivate and outline it in Section \ref{Construction of an Adjustment Term}.
\end{remark}

\subsection{Further results and discussion of related problems}\label{Sec:FurhterResults}
For Equation \eqref{qwe}, a family of self-similar solutions has been constructed by Dai and Duyckaerts \cite{DD21} in $d=7$,  with radial profiles being strictly positive and globally smooth.  However, the stability of these solutions is unknown. Also, we note that this family does not include $u_T^*$ as the profile $U$ obviously changes sign. In $d \geq 8$, our solution is so far the only non-trivial example of self-similar blowup for Equation \eqref{qwe}, although a similar result as in \cite{DD21} is expected to hold at least in a certain range of dimensions. Moreover, for $d$ sufficiently  large, the existence of non self-similar blowup solutions is anticipated, see the work of Collot \cite{Collot2018}.

As outlined above, our approach for proving conditional asymptotic stability of self-similar solutions using hyperboloidal similarity coordinates hinges critically on the spectral problem, which in turn reduces to the analysis of an ODE. Using the corresponding results of \cite{GS21} our methods can be used to obtain an analogue of Theorem \ref{stability on whole space} for the radial wave equation with cubic nonlinearity 
\begin{align}\label{Eq:NLW3}
\Big(\partial_t^2-\partial_r^2-\frac{d-1}{r}\partial_r\Big)u(t,r)=  u(t,r)^3,
\end{align}
in $d=7$ and the corresponding self-similar blowup solution
\begin{align*}
u_T^{cub}(t,r) = \frac{4(T-t)}{(T-t)^2+r^2},
\end{align*}
which is known to be conditionally (co-dimension one) stable in backward light cones.

We note that Equations \eqref{qwe} and \eqref{Eq:NLW3} can be viewed as toy models for the radial Yang-Mills Equation on $\R^{1,d}$ ,
\begin{align}\label{Eq:YM}
\Big(\partial_t^2-\partial_r^2-\frac{d+1}{r}\partial_r\Big)u(t,r)=  (d-4) u(t,r)^2  ( 3  - r^2 u(t,r)),
\end{align}
which is supercritical in $d \geq 5$, and for supercritical co-rotational wave maps from $\R^{1,d}$ to $\mathbb S^d$ for $d \geq 3$,  governed by 
\begin{align}\label{Eq:WM}
\Big(\partial_t^2-\partial_r^2-\frac{d+1}{r}\partial_r\Big)u(t,r)= u(t,r)^3 G(r,u(t,r)),
\end{align}
with $G$ a globally bounded function, non-zero at the origin. For both problems stable self-similar blowup is well-known, see \cite{BDS19} and \cite{DO21} for corresponding results in hyperboloidal similarity coordinates. More interestingly, unstable self-similar solutions have been observed  numerically as intermediate attractors close to the threshold for singularity formation, for Equation \eqref{Eq:YM} in $d=5$ by Bizo\'n \cite{Bizon2002} and for Equation \eqref{Eq:WM} in $3 \leq d \leq 6$ by Biernat, Bizo\'n, and Maliborski \cite{BiernatBizonMaliborski2016}. These \textit{critical} self-similar solutions, which are not known in closed form, are supposed to have exactly one genuine unstable direction. From an analytic point of view, the threshold problem is completely open for energy supercritical wave equations. This further motivates the study of simpler toy models such as Equation \eqref{qwe}. In fact, our results on the conditional stability of $u_T^*$ support the conjecture that this solution plays an important role in the study of threshold phenomena for Equation \eqref{qwe}.

\subsection{Notation and Conventions}
We denote by $\mathbb N$, $\mathbb Z$, $\mathbb R$, and $\mathbb C$ the sets of natural numbers, integers, real numbers, and complex numbers respectively. By $\mathbb N_0:=\mathbb N\cup\{0\}$ we denote the nonnegative integers. Furthermore, we will denote by $\mathbb R_{+}:=(0,\infty)$. Given a complex number $z\in\mathbb C$, we denote by $\Re z$ and $\Im z$ its real and imaginary parts respectively. We denote by $\mathbb H:=\{z\in\mathbb C:\Re z>0\}$ the open right-half complex plane. Given $R>0$ and $d\in\mathbb N$, we denote by $\mathbb B_R^d:=\{x\in\mathbb R^d:|x|<R\}$ the open ball in $\mathbb R^d$ of radius $R$ centered at the origin. 

Given $x,y\in\mathbb R_+$, we say $x\lesssim y$ if there exists a constant $C>0$ such that $x\leq Cy$. Furthermore, we say that $x\simeq y$ if $x\lesssim y$ and $y\lesssim x$. For a one-parameter family of positive numbers $x_\lambda,y_\lambda$, we say that $x_\lambda\lesssim y_\lambda$ if there exists a constant $C>0$, independent of the parameter $\lambda$, such that $x_\lambda\leq Cy_\lambda$ for all $\lambda$.

On a Hilbert space $\mathcal H$, we denote by $\mathcal B(\mathcal H)$ the space of bounded linear operators. For a closed operator $\big(L,\mathcal D(L)\big)$ on the Hilbert space $\mathcal H$ with domain $\mathcal D( L)$, we denote the resolvent set by $\rho(L)$ and by $R_L(\lambda):=(\lambda I- L)^{-1}$ the resolvent operator for $\lambda\in\rho( L)$. Furthermore, we denote by $\sigma( L):=\mathbb C\setminus\rho( L)$ the spectrum of $\big(L,\mathcal D(L)\big)$. In particular we denote by $\sigma_p( L):=\{\lambda\in\sigma(L):\exists  u\in\mathcal D( L)\setminus\{0\}\text{ such that }  u\in\ker(\lambda I- L)\}$ the point spectrum of $\big(L,\mathcal D(L)\big)$. As we will only work with \textit{strongly continuous semigroups $\big(S(s)\big)_{s\geq0}$ of bounded operators on} $\mathcal H$, we will instead refer to these simply as \textit{semigroups on} $\mathcal H$ whenever necessary.
		
On a domain $\Omega\subset\mathbb R$ and given functions $f,g\in C^1(\Omega)$, we denote their Wronskian by $W(f,g)(x)=f(x)g'(x)-f'(x)g(x)$. Furthermore, for a function of multiple variables, we define $u[t_0]:=\big(u(t_0,\cdot),\partial_0u(t_0,\cdot)\big)$ where $t_0\in\mathbb R$.


\subsection{Function spaces} \label{Functional Setting}
For $k\in\mathbb N_0$ and $U\subset\mathbb R^d$ open and bounded, we define  Sobolev norms by 
	\[
				\|f\|^2_{H^k(U)}:=\sum_{|\kappa|\leq k}\|\partial^\kappa f\|^2_{L^2(U)},\;\|f\|^2_{\dot{H}^k(\Omega)}:=\sum_{|\kappa|=k}\|\partial^\kappa f\|^2_{L^2(U)}
	\]
		for all $f\in C^\infty(\overline U)$ where $\kappa=(\kappa_1,\dots,\kappa_d)\in\mathbb N^d$ is a multi-index with $|\kappa|:=\sum_{i=1}^d\kappa_i$. For $k \in \N_0$, the Sobolev space $H^k(U)$ is then defined as the completion of $C^\infty(\overline U)$ with respect to the norm $\|\cdot\|_{H^k(U)}$. 
Since we restrict ourselves to the radial setting, we define for 
$R>0$, $k\in\mathbb N_0$, and $d\in\mathbb N$, the following space of functions
			$$
				H^k_\text{rad}(\mathbb B^d_R):=\{f:(0,R)\to\mathbb C:f(|\cdot|)\in H^k(\mathbb B^d_R)\}.
			$$
		We set
			$$
				\mathcal H^k_R:=H^k_\text{rad}(\mathbb B^d_R)\times H^{k-1}_\text{rad}(\mathbb B^d_R)
			$$
		which is a Hilbert space with inner product 
			$$
				(\mathbf f|\mathbf g)_{\mathcal H^k_R}:= \left (f_1(|\cdot|)|g_1(|\cdot|) \right )_{H^k(\mathbb B^d_R)}+(f_2(|\cdot|)|g_2(|\cdot|))_{H^{k-1}(\mathbb B^d_R)},
			$$
for  $\mathbf f=(f_1,f_2)$ and $\mathbf g=(g_1,g_2)$ and norm
			$$
				\|\mathbf f\|_{\mathcal H^k_R}^2:= (\mathbf f|\mathbf f)_{\mathcal H^k_R} = \|f_1(|\cdot|)\|_{H^k(\mathbb B^d_R)}^2+\|f_2(|\cdot|)\|_{H^{k-1}(\mathbb B^d_R)}^2
			$$
We also consider the space of radial functions which are smooth up to the boundary of $\overline{\mathbb B_R^d}$
			$$
				C^\infty_\text{rad}(\overline{\mathbb B_R^d}):=\{f\in C^\infty(\overline{\mathbb B_R^d}):f\text{ is radial}\}
			$$
		and the space of smooth ``even'' functions
			$$
				C_e^\infty[0,R]:=\{f\in C^\infty[0,R]:f^{(2j+1)}(0)=0\text{ for }j\in\mathbb N_0\}.
			$$
		By Lemma 2.1 of \cite{G21}, there is a one-to-one correspondence between $C^\infty_\text{rad}(\overline{\mathbb B_R^d})$ and $C_e^\infty[0,R]$. For ease of reading, we attempt to avoid switching between $C^\infty_\text{rad}(\overline{\mathbb B_R^d})$ and $C_e^\infty[0,R]$ and stick only with $C_e^\infty[0,R]$ whenever possible. We remark that $C_e^\infty[0,R]$ is dense in $H^k_\text{rad}(\mathbb B^d_R)$ which implies $C_e^\infty[0,R]^2$ is dense in $\mathcal H_R^k$.

		\subsection{Short overview of the paper}
Establishing Theorem \ref{stability on whole space} proceeds in two steps. First, we reformulate the equation for small perturbations of $u_T^*$ as an abstract evolution equation in hyperboloidal similarity coordinates 
				\begin{align}\label{Eq:Hyperbolic}
					\begin{cases}
						\partial_s\Phi(s)=( \mathbf L_0 + \mb L')\Phi(s) + \mb N(\Phi(s)),
						\\
						\Phi(0)=\Phi_0,
					\end{cases}
				\end{align}
where $\mb L_0$ represents the free wave evolution, $\mb L'$ the perturbation arising from linearization around $u_T^*$, and $\mb N$ the remaining nonlinearity. For precise definitions of these objects, see Section \ref{The Quadratic Wave Equation in Hyperboloidal Similarity Coordinates}. In Section \ref{Linear Stability Analysis}, we study the operator $\mb L = \mb L_0 + \mb L'$ on the radial Sobolev space $\mc H^6_R$ for $R \geq \frac12$ fixed. By using the results of \cite{DO21} on the free operator $\mb L_0$, which we summarize in  Section \ref{The Wave Equation in Hyperboloidal Similarity Coordinates}, it is straightforward to infer that $\mathbf L$ generates a semigroup $\big(\mathbf S(s)\big)_{s\geq0}$ on $\mc H^6_R$. A  careful spectral analysis carried out in Section \ref{Spectral Analysis_qwe} reveals that there is an $\omega_0 > 0$ such that 
\[ \sigma(\mb L) \subset \{ \lambda \in \C: \mathrm{Re} \lambda \leq - \omega_0 \} \cup \{1,4\}. \]
Consequently, according to a spectral mapping theorem that applies to our setting, we prove linear stability for a co-dimension two subspace of initial data in $\mathcal H^6_R$. In Section \ref{Nonlinear Stability Analysis}, we study Equation \eqref{Eq:Hyperbolic} via Duhamel's formula. That is, we study the integral equation
\[ \Phi(s)=\mathbf S(s)\Phi_0+\int_0^s\mathbf S(s-s')\mathbf N\big(\Phi(s')\big)ds'. \]
By removing the unstable contributions of the initial data, we prove the existence of an  exponentially decaying solution by a standard fixed point argument. Though we do not carry it out explicitly, this can be formulated as saying that there is a co-dimension two manifold of initial data on an initial hyperboloid which leads to blowup via $u^*_T$ along hyperboloidal time-slices as $s \to \infty$. 

The second step in the proof of Theorem \ref{stability on whole space}, carried out in Sections \ref{Preparation of Hyperboloidal Initial Data} and \ref{Evolution of Physical Initial Data}, consists of constructing the adjustment term $(Y_1,Y_2)$ and evolving initial data from $t =0$ of the form 
\[ u_1^*[0]+(f,g)+\alpha(Y_1,Y_2)\]
for arbitrary, small $(f,g)\in C^{\infty}(\B^7)^2$ and  $\alpha\in\mathbb R$. By properly choosing the blowup time $T$ and the parameter $\alpha$, we find a hyperboloid to which we can restrict the physical evolution of our data in a way that allows us to continue its evolution in a hyperboloidal region (see Figure \ref{HSC_Spacetime_Diagram}) via Equation \eqref{Eq:Hyperbolic}. From this, we infer the convergence claimed in Equation \eqref{convergence}.

	\section{The Wave Equation in Hyperboloidal Similarity Coordinates} \label{The Wave Equation in Hyperboloidal Similarity Coordinates}
		In this section, we reformulate Equation \eqref{qwe} as a first-order system in hyperboloidal similarity coordinates. First, we review the corresponding well-posedness theory for the free radial wave equation developed in \cite{DO21} by Donninger and Ostermann.
		
		\subsection{Free Wave Evolution in Hyperboloidal Similarity Coordinates} \label{Free Wave Evolution in Hyperboloidal Similarity Coordinates}
			Unless otherwise stated, all functions are assumed to be radial. Let $d\in\mathbb N$ and $u\in C^\infty(\mathbb R\times[0,\infty))$. With $v=u\circ\eta_T$, we infer by the chain rule that the quantity
				$$
					\Big(\partial_t^2-\partial_r^2-\frac{d-1}{r}\partial_r\Big)u(t,r)
				$$
			transforms into
				$$
					-g^{00}(s,y)\Big(\partial_s^2-c_{11}^d(y)\partial_y-c_{12}(y)\partial_y^2-c_{20}^d(y)\partial_s-c_{21}(y)\partial_y\partial_s\Big)v(s,y),
				$$
			where
				\begin{align*}
					g^{00}(s,y)=&-e^{2s}\frac{1-h'(y)^2}{\big(yh'(y)-h(y)\big)^2},
					\\
					c_{11}^d(y)=&-\frac{d-1}{y}\frac{\big(y  h'(y)- h(y)\big) h(y)}{1- h'(y)^2}+\frac{y^2- h(y)^2}{1- h'(y)^2}\frac{y  h''(y)}{y  h'(y)- h(y)}+2\frac{ h(y) h'(y)-y}{1- h'(y)^2},
					\\
					c_{12}(y)=&\frac{ h(y)^2-y^2}{1- h'(y)^2},
					\\
					c_{20}^d(y)=&-1-\frac{d-1}{y}\frac{\big(y  h'(y)- h(y)\big) h'(y)}{1- h'(y)^2}+\frac{y^2- h(y)^2}{1- h'(y)^2}\frac{ h''(y)}{y  h'(y)- h(y)},
					\\
					c_{21}(y)=&2\frac{ h(y) h'(y)-y}{1- h'(y)^2}.
				\end{align*}
			
			\begin{definition}
				Let $R>0$ and $d,k\in\mathbb N$. We define the \textit{free radial wave evolution} as the unbounded operator $\big(\tilde{\mathbf L}_d,\mathcal D(\tilde{\mathbf L}_d)\big)$, with $\mathcal D(\tilde{\mathbf L}_d)=C^\infty_e[0,R]^2$, on $\mathcal H_R^k$ by
					$$
						\tilde{\mathbf L}_d\mathbf f(y):=
						\begin{pmatrix}
							f_2(y)
							\\
							c_{11}^d(y)f'_1(y)+c_{12}(y)f''_1(y)+c_{20}^d(y)f_2(y)+c_{21}(y)f'_2(y)
						\end{pmatrix}.
					$$
			\end{definition}
			
			The linear, radial wave equation in hyperboloidal similarity coordinates is equivalent to the first-order system
				$$
					\partial_s\mathbf v(s,\cdot)=\tilde{\mathbf L}_d\mathbf v(s,\cdot)
				$$
			where
				$$
					\mathbf v(s,\cdot)=\begin{pmatrix}v(s,\cdot)\\\partial_ sv(s,\cdot)\end{pmatrix}.
				$$
			In \cite{DO21}, it was shown that $\big(\tilde{\mathbf L}_d,\mathcal D(\tilde{\mathbf L}_d)\big)$ is closable and its closure, $\big(\mathbf L_d,\mathcal D(\mathbf L_d)\big)$, generates a semigroup which we recall here.
				\begin{lemma}[\cite{DO21}, Theorem 2.1] \label{generation}
					Let $R\geq\frac{1}{2}$, $d,k,\in\mathbb N$, such that $d\geq3$ is odd and $k\geq\frac{d-1}{2}$. The operator $\big(\tilde{\mathbf L}_d,\mathcal D(\tilde{\mathbf L}_d)\big)$ is closable and its closure, $\big(\mathbf L_d,\mathcal D(\mathbf L_d)\big)$, is the generator of a semigroup $\big(\mathbf S_d(s)\big)_{s\geq0}$ on $\mathcal H_R^k$ with the property that there exists $M\geq1$ such that
						$$
							\|\mathbf S_d(s)\mathbf f\|_{\mathcal H_R^k}\leq Me^{\frac{s}{2}}\|\mathbf f\|_{\mathcal H_R^k}
						$$
					for all $s\geq0$ and $\mathbf f\in\mathcal H_R^k$.
				\end{lemma}

		\subsection{The Quadratic Wave Equation in Hyperboloidal Similarity Coordinates} \label{The Quadratic Wave Equation in Hyperboloidal Similarity Coordinates}
			Now, we can reformulate Equation \eqref{qwe} as a first-order system in hyperboloidal similarity coordinates. For the remainder of this paper, we fix $d=7$. We look for solutions of the form $u=u_T^*+\tilde u$ where $\tilde u$ represents some perturbation of $u_T^*$ with $T>0$ to be determined later. With this ansatz, Equation \eqref{qwe} becomes
				$$
					\bigg(\partial_t^2-\partial_r^2-\frac{6}{r}\partial_r+V_T(t,r)\bigg)\tilde u(t,r)=\tilde u(t,r)^2
				$$
			where $V_T:=-2u_T^*$. Setting $\tilde v=\tilde u\circ\eta_T$, we obtain the equation
				\begin{align*}
					\partial_s^2\tilde v(s,y)=&c_{11}(y)\partial_y\tilde v(s,y)+c_{12}(y)\partial_y^2\tilde v(s,y)+c_{20}(y)\partial_s\tilde v(s,y)+c_{21}(y)\partial_y\partial_s\tilde v(s,y)
					\\
					&+\frac{(V_T\circ\eta_T)(s,y)}{g^{00}(s,y)}\tilde v(s,y)-\frac{\tilde v(s,y)^2}{g^{00}(s,y)}
				\end{align*}
			where $c_{11}:=c_{11}^7$ and $c_{20}:=c_{20}^7$. Observe that the function 
				$$
					V(y):=\frac{V_T\big(\eta_T(s,y)\big)}{g^{00}(s,y)}=\frac{48\big(21h(y)^2-5y^2\big)}{\big(5y^2+3h(y)^2\big)^2}\frac{\big(yh'(y)-h(y)\big)^2}{\big(1-h'(y)^2\big)}.
				$$
			is in $C_e^\infty[0,R]$ for any $R>0$. Furthermore, we write
				$$
					-\frac{\tilde v(s,y)^2}{g^{00}(s,y)}=e^{2s}N(y,e^{-2s}\tilde v)
				$$
			where
				$$
					N(y,x):=\frac{\big(yh'(y)-h(y)\big)^2}{1-h'(y)^2}x^2.
				$$
			Upon setting
				$$
					\tilde{\mathbf v}(s,\cdot):=\begin{pmatrix}\tilde v(s,\cdot)\\\partial_ s\tilde v(s,\cdot)\end{pmatrix},
				$$
			the quadratic wave equation, as a first-order system in hyperboloidal similarity coordinates, takes the form
				$$
					\partial_s\tilde{\mathbf v}(s,\cdot)=(\tilde{\mathbf L}_7+\mathbf L')\tilde{\mathbf v}(s,\cdot)+e^{2s}\mathbf N\big(e^{-2s}\tilde{\mathbf v}(s,\cdot)\big)
				$$
			where $\mathbf L'$ is defined by
				$$
					\mathbf L'\mathbf f(y):=
					\begin{pmatrix}
						0
						\\
						V(y)f_1(y)
					\end{pmatrix}
				$$
			and
				$$
					\mathbf N\big(\mathbf f\big)(y):=
					\begin{pmatrix}
						0
						\\
						N(y,f_1(y))
					\end{pmatrix}.
				$$
			As $V\in C_e^\infty[0,R]$ for any $R>0$, we see that $\mathbf L'\in\mathcal B(\mathcal H_R^k)$ for any $R>0$ and $k\in\mathbb N$. An autonomous equation is obtained by setting $\Phi(s):=e^{-2s}\tilde{\mathbf v}(s,\cdot)$ which yields
				\begin{equation}
					\partial_s\Phi(s)=(\tilde{\mathbf L}_7-2\mathbf I+\mathbf L')\Phi(s)+\mathbf N\big(\Phi(s)\big). \label{NL eqn}
				\end{equation}
			In what follows, we set 
			
\[\tilde{\mathbf L}:=\tilde{\mathbf L}_7-2\mathbf I+\mathbf L'\] in which case $\big(\tilde{\mathbf L},\mathcal D(\tilde{\mathbf L})\big)$ is an unbounded, densely defined operator on $\mathcal H_R^k$ with $\mathcal D(\tilde{\mathbf L}):=\mathcal D(\tilde{\mathbf L}_7)$ for $R\geq\frac{1}{2}$. From this point on, we refrain from referring to the domains of the various operators unless absolutely necessary. Furthermore, $R$ will always denote an arbitrary real number satisfying $R\geq\frac{1}{2}$. To simplify notation, we set $\mathcal H_R:=\mathcal H_R^6$.

\section{Linear Stability Analysis} \label{Linear Stability Analysis}
		\subsection{Well-Posedness of the Linearized Evolution} \label{Analysis of the Linearized Evolution}
			First, we show that $\tilde{\mathbf L}$ is closable and its closure, $\mathbf L$, is the generator of a semigroup $\big(\mathbf S(s)\big)_{s\geq0}$ on $\mathcal H_R$. In fact, this is a very simple consequence of Lemma \ref{generation}.
				\begin{lemma} \label{well-posed}
					The operator $\tilde{\mathbf L}$ is closable and its closure, denoted by $\mathbf L$, is the generator of a semigroup $\big(\mathbf S(s)\big)_{s\geq0}$ on $\mathcal H_R$ and satisfies the estimate
						\begin{equation}
							\|\mathbf S(s)\mathbf f\|_{\mathcal H_R}\leq Me^{\big(-\frac{3}{2}+M\|\mathbf L'\|_{\mathcal H_R}\big)s}\|\mathbf f\|_{\mathcal H_R} \label{qwe growth bound}
						\end{equation}
					for all $s\geq0$, $\mathbf f\in\mathcal H_R$ and for $M\geq1$ as in Lemma \ref{generation}.
				\end{lemma}
				\begin{proof}
					From Lemma \ref{generation}, we infer the existence of a semigroup $\big(\mathbf S_7(s)\big)_{s\geq0}$ generated by $\mathbf L_7$ on $\mathcal H_R$ satisfying the estimate $\|\mathbf S_7(s)\|_{\mathcal H_R}\leq Me^{\frac{s}{2}}$ for some $M\geq1$ and all $s\geq0$. As a consequence, the operator $\mathbf L_7-2\mathbf I$, with $\mathcal D(\mathbf L_7-2\mathbf I)=\mathcal D(\mathbf L_7)$, generates the semigroup $\big(\mathbf S_0(s)\big)_{s\geq0}$ on $\mathcal H_R$ given by $\mathbf S_0(s)=e^{-2s}\mathbf S_7(s)$ which satisfies $\|\mathbf S_0(s)\|_{\mathcal H_R}\leq Me^{-\frac{3}{2}s}$ for all $s\geq0$. Since $\mathbf L'\in\mathcal B(\mathcal H_R)$, the bounded perturbation theorem (see \cite{EN00}, Theorem III.1.3) implies that $\mathbf L$ generates a semigroup $\big(\mathbf S(s)\big)_{s\geq0}$ on $\mathcal H_R$ satisfying the claimed estimate.
				\end{proof}
				
		\subsection{Spectral Analysis} \label{Spectral Analysis_qwe}	
			Observe that Lemma \ref{well-posed} does not necessarily exclude exponential growth of the semigroup. More precisely, \eqref{qwe growth bound} implies that the growth bound for the semigroup $\big(\mathbf S(s)\big)_{s\geq0}$ is at most $-\frac{3}{2}+M\|\mathbf L'\|_{\mathcal H_R}$. Without further information, the sign of this upper bound could be positive. To conclude our analysis of the linear evolution, it will be necessary to improve this upper bound. The improvement we seek would involve showing that, for a sufficiently large subspace of $\mathcal H_R$, this upper bound is indeed negative. In fact, since $V\in C_e^\infty[0,R]$ and the embedding $H_\text{rad}^6(\mathbb B^7)\hookrightarrow H_\text{rad}^5(\mathbb B^7)$ is compact, we infer that $\mathbf L'$ is a compact operator on $\mathcal H_R$. Thus, we can apply Theorem B.1 of \cite{G21} to obtain the desired improvement provided we have a sufficient characterization of $\sigma(\mathbf L)$. According to the following lemma, we can restrict our attention to understanding $\sigma_p(\mathbf L)$.
				\begin{lemma} \label{restrict to point spectrum}
					Let $\epsilon>0$. The set $S_\epsilon:=\sigma(\mathbf L)\cap\{\lambda\in\mathbb C:\Re\lambda\geq-\frac{3}{2}+\epsilon\}$ consists of finitely many eigenvalues of $\mathbf L$, all of which have finite algebraic multiplicity.
				\end{lemma}
				\begin{proof}
					Theorem II.1.10.ii of \cite{EN00} implies that $\sigma(\mathbf L_7-2\mathbf I)\subseteq\{\lambda\in\mathbb C:\Re\lambda\leq-\frac{3}{2}\}$. Since $\mathbf L'$ is compact, Theorem B.1.i of \cite{G21} implies the claim.
				\end{proof}
			
			By Theorem B.1.ii and B.1.iii of \cite{G21}, a characterization of the unstable portion of the point spectrum, namely $\sigma_p(\mathbf L)\cap\overline{\mathbb H}$, is sufficient to obtain an improvement of \eqref{qwe growth bound} on the remaining stable subspace. This is achieved by the following key proposition.
			
				\begin{proposition} \label{spectrum}
					We have that
						$$
							\sigma_p(\mathbf L)\cap\overline{\mathbb H}=\{1,4\}.
						$$
					Furthermore, $\ker(\mathbf I-\mathbf L)=\langle\mathbf f_1^*\rangle$ and $\ker(4\mathbf I-\mathbf L)=\langle\mathbf f_4^*\rangle$ where
						$$
							\mathbf f_1^*(y):=\frac{(7h(y)^2-15y^2)h(y)}{(5y^2+3h(y)^2)^3}
								\begin{pmatrix}
									1
									\\
									3
								\end{pmatrix},\;\;
							\mathbf f_4^*(y):=\frac{1}{(5y^2+3h(y)^2)^3}
								\begin{pmatrix}
									1
									\\
									6
								\end{pmatrix}.
						$$
				\end{proposition}
				\begin{proof}
					For $\lambda=1,4$, direct calculations verify that $\mathbf f_\lambda^*\in\mathcal D(\tilde{\mathbf L})$ and $\langle\mathbf f_\lambda^*\rangle\subseteq\ker(\lambda\mathbf I-\mathbf L)$. The reverse inclusion follows from simple ODE arguments. In particular, $\{1,4\}\subseteq\sigma_p(\mathbf L)\cap\overline{\mathbb H}$.
					
					
					Now, we aim to show $\sigma_p(\mathbf L)\cap\overline{\mathbb H}\subseteq\{1,4\}$. This direction of the argument is highly nontrivial and constitutes one of the major novelties of this paper. In an effort to aid the reader, we will take a moment to summarize the remainder of this argument before proceeding. The first step is to establish a connection between the existence of eigenfunctions of $\mathbf L$ and the existence of analytic solutions of a particular ODE, where it suffices to restrict our attention to a backward light cone. After having established this connection, we transform this equation into another `supersymmetric' problem, where the solutions corresponding to the known eigenvalues $\lambda=1,4$ transform to trivial solutions. As will be seen, there is a correspondence between analytic solutions of both ODEs. The third step is to then analyze this new equation and show that it does not have any nontrivial analytic solutions.  As a consequence of the first step, we are able to then exclude the existence of eigenvalues with the exception of $\lambda=1,4$.\\
					
					\textit{Step 1: Reduction to an ODE Problem.} We argue by contradiction. Suppose $\lambda\in\sigma_p(\mathbf L)\cap\overline{\mathbb H}$, $\lambda \neq 1,4$. Thus, there exists $\mathbf f_\lambda=(f_{\lambda,1},f_{\lambda_2})\in\mathcal D(\mathbf L)\setminus\{\mathbf 0\}$ with $(\lambda\mathbf I-\mathbf L)\mathbf f_\lambda=\mathbf 0$. A direct calculation shows that $f_{\lambda,1}$ solves the ODE
						\begin{equation}
				 			f_{\lambda,1}''(y)+\frac{c_{11}(y)+(\lambda+2)c_{21}(y)}{c_{12}(y)}f_{\lambda,1}'(y)+\frac{(\lambda+2)(c_{20}(y)-\lambda-2)+ V(y)}{c_{12}(y)}f_{\lambda,1}(y)=0 \label{homogeneous spectral equation}
						\end{equation}
					weakly on the interval $(0,R)$ and $f_{\lambda,2}=(\lambda+2)f_{\lambda,1}$. Furthermore, since $f_{\lambda,1}\in H^6_\text{rad}(\mathbb B^7)$, Sobolev embedding implies $f_{\lambda,1}\in C^2(0,R)$. Thus, $f_{\lambda,1}$ is a classical solution of Equation \eqref{homogeneous spectral equation} on $(0,R)$. We transform this equation into standard form. Thereby, we restrict ourselves to $0 \leq y \leq \frac{1}{2}$ and use the correspondence between hyperboloidal similarity coordinates and standard similarity coordinates inside the backward light cone. We recall that standard similarity coordinates are defined via the map
						$$
							(\tau,\rho)\mapsto\big(T-e^{-\tau},e^{-\tau}\rho\big)
						$$
					which maps the infinite cylinder $\mathbb R\times[0,1]$ into the backward light cone with vertex $(T,0)$. First,  $v(s,y):=e^{(\lambda+2)s}f_{\lambda,1}(y)$ is in $C^2(\mathbb R\times(0,R))$ and is a classical solution of the equation
						$$
							\partial_s^2 v(s,y)=c_{11}(y)\partial_y v(s,y)+c_{12}(y)\partial_y^2 v(s,y)+c_{20}(y)\partial_s v(s,y)+c_{21}(y)\partial_y\partial_s v(s,y)+V(y)v(s,y)
						$$
					for $(s,y)\in\mathbb R \times(0,R)$. Upon setting
						$$
							V_0(y)=-\frac{48(21-5y^2)}{(5y^2+3)^2},
						$$
					and defining $v(s,y)=:w\Big(s-\log\big(-h(y)\big),-\dfrac{y}{h(y)}\Big)$ for $y \in (0,\frac{1}{2})$, we find that $w$ is a classical solution of the equation
						$$
							\Big(\partial_\tau^2+2\rho\partial_\rho\partial_\tau-\big(1-\rho^2\big)\partial_\rho^2-\frac{6}{\rho}\partial_\rho+\partial_\tau+2\rho\partial_\rho+V_0(\rho)\Big)w(\tau,\rho)=0
						$$
					on $\mathbb R \times(0,1)$. In terms of $f_{\lambda,1}$, we have
						\begin{align*}
							w(\tau,\rho)&=e^{(\lambda+2)\tau}\Big(\frac{2}{2+\sqrt{2(1+\rho^2)}}\Big)^{\lambda+2}f_{\lambda,1}\Big(\frac{2\rho}{2+\sqrt{2(1+\rho^2)}}\Big) =:e^{(\lambda+2)\tau}f(\rho).
						\end{align*}
					Thus, $f$ is a classical solution of the ODE
						\begin{equation}
							(1-\rho^2)f''(\rho)+\Big(\frac{6}{\rho}-2(\lambda+3)\rho\Big)f'(\rho)-\Big((\lambda+2)(\lambda+3)-\frac{48(21-5\rho^2)}{(5\rho^2+3)^2}\Big)f(\rho)=0. \label{qwe spectral ode}
						\end{equation}
					Smoothness of the coefficients implies $f\in C^\infty(0,1)$. Observe that $\rho=0$ is a regular singular point of Equation \eqref{qwe spectral ode} with Frobenius indices $\{0,-5\}$ and so is $\rho=1$ with Frobenius indices $\{0,1-\lambda\}$. The Frobenius analysis in the proof of Proposition 3.2 of \cite{G21} with $d=5$ allows us to conclude that $f\in C^\infty[0,1]$. Now, our goal is to show that for $\lambda \neq 1,4$, Equation \eqref{qwe spectral ode} does not have solutions in $C^\infty[0,1]$.\\
					
					\textit{Step 2: Supersymmetric Removal.} First, observe that the functions 
						\begin{align*}
							f(\rho;1):=\frac{7-15\rho^2}{(5\rho^2+3)^3}, \quad 
							f(\rho;4):=\frac{1}{(5\rho^2+3)^3}
						\end{align*}
					are indeed solutions in $C^\infty[0,1]$ with $\lambda=1$ and $\lambda=4$ respectively. To investigate $C^\infty[0,1]$ solutions of Equation \eqref{qwe spectral ode} with $\lambda\neq1,4$, we first `remove' the eigenvalues $\lambda=1$ and $\lambda=4$ by performing so-called \textit{supersymmetric} transformation. For an in-depth discussion of this procedure, we refer the reader to \cite{GS21}, Appendix B and \cite{G18}, Section 2.5. We begin by making the change of variables
					$$
						f(\rho)=:\rho^{-3}(1-\rho^2)^{-\frac{\lambda}{2}}g(\rho)
					$$
				which transforms Equation \eqref{qwe spectral ode} into
					\begin{equation}
						-g''(\rho)-\frac{2 \left(95\rho^6-729\rho^4+405\rho^2-27\right)}{\rho^2 \left(1-\rho^2\right)^2
   \left(5\rho^2+3\right)^2}g(\rho)=-\frac{(\lambda+2) (\lambda -4)}{\left(1-\rho^2\right)^2}g(\rho). \label{first transform qwe spectral ode}
					\end{equation}
				Consequently, $g(\rho;4)=\rho^3(1-\rho^2)^2f(\rho;4)$ is a solution of Equation \eqref{first transform qwe spectral ode} with $\lambda=4$. Our goal is to factor the left-hand side of Equation \eqref{first transform qwe spectral ode} using the solution $g(\rho;4)$. Following the standard procedure, see the above mentioned references,  the left-hand side can be factored as
					\begin{align*}
						-\partial_\rho^2&-\frac{2 \left(95\rho^6-729\rho^4+405\rho^2-27\right)}{\rho^2 \left(1-\rho^2\right)^2\left(5\rho^2+3\right)^2}=\Big(-\partial_\rho-b(\rho) \Big)\Big(\partial_\rho-b(\rho)\Big).
					\end{align*}
where $b(\rho):= \frac{9-36\rho^2 -5\rho^4}{3\rho+2\rho^3-5\rho^5}$. Setting $\tilde g:=g'-b g$ and defining $\tilde g(\rho)=:\rho^3(1-\rho^2)^{\frac{\lambda}{2}}\tilde f(\rho)$ produces the new equation
					$$
						(1-\rho^2)\tilde f''(\rho)+\Big(\frac{6}{\rho}-2(\lambda+3)\rho\Big)\tilde f'(\rho)-\Big((\lambda+2)(\lambda+3)-\frac{18 \left(5\rho^4+30 \rho^2-3\right)}{\rho^2 \left(5\rho^2+3\right)^2}\Big)\tilde f(\rho)=0. 
					$$
				Observe that for $\lambda =4$, the above  transformations yield $\tilde f(\rho;4) = 0$. In this sense, we have `removed' the eigenvalue $\lambda=4$ by transforming the corresponding solution, $f(\rho;4)$, into the trivial solution. For $\lambda = 1$, we obtain the solution $\tilde f(\rho;1):=-\frac{3\rho}{(3+5\rho^2)^2}$. Repeating the same transformations but with the factorization given by the solution $\tilde f(\rho;1)$ instead produces the new equation
					\begin{equation}
						(1-\rho^2)\hat f''(\rho)+\Big(\frac{6}{\rho}-2(\lambda+3)\rho\Big)\hat f'(\rho)-\Big((\lambda+2)(\lambda+3)-\frac{6 \left(35 \rho^4+18 \rho^2-21\right)}{\rho^2 \left(5 \rho^2+3\right)^2}\Big)\hat f(\rho)=0 \label{susy spectral eqn}
					\end{equation}
				for the corresponding new dependent variable $\hat f$. \\
					
				\textit{Step 3: Analysis of Eq.~\eqref{susy spectral eqn}.} Now, we show that Equation \eqref{susy spectral eqn} has no non-trivial analytic solutions for $\lambda\in\overline{\mathbb H}$. We achieve this by expanding any nontrivial, analytic solution around the regular singular point $\rho=0$ and showing that if $\lambda\in\overline{\mathbb H}$, then this solution cannot be analytically continued past $\rho=1$. 
				
				Observe that Equation \eqref{susy spectral eqn} has seven regular singular points: $\rho=0,\pm1,\pm i\sqrt{\frac{3}{5}},$ and $\pm\infty$. We begin our analysis by first reducing the number of regular singular points to four via the transformation
					$$
						\rho=\sqrt{\frac{3x}{8-5x}},\quad \hat f(\rho)=x(8-5x)^{\frac{\lambda+2}{2}}y(x)
					$$
				which transforms Equation \eqref{susy spectral eqn} into its Heun form
					\begin{equation}
						y''(x)+\Big(\frac{11}{2x}+\frac{\lambda}{x-1}+\frac{1}{2(x-\frac{8}{5})}\Big)y'(x)+\frac{5(\lambda+2)(\lambda+8)x-(\lambda+26)(3\lambda+4)}{20x(x-1)(x-\frac{8}{5})}y(x)=0 \label{heun spectral eqn}
					\end{equation}
				with the four regular singular points $x=0,1,\frac{8}{5},\infty$. Frobenius theory implies that any $y\in C^\infty[0,1]$ solving Equation \eqref{heun spectral eqn} is analytic on $[0,1]$. In addition, any analytic solution of Equation \eqref{heun spectral eqn} yields an analytic solution of Equation \eqref{susy spectral eqn} as well as the converse. Thus, to exclude the existence of analytic solutions of Equation \eqref{susy spectral eqn}, we exclude the existence of analytic solutions of Equation \eqref{heun spectral eqn}. Thereby, we apply a similar strategy as in  \cite{CDGH15}, \cite{G18}, \cite{GS21} and \cite{CGS21}.
				
				At $x=0$, the Frobenius indices are $\{0,-\frac{9}{2}\}$. Without loss of generality, we may assume that a solution for a fixed $\lambda$, denoted by $y(\cdot;\lambda)$, has the expansion
					\begin{equation}
						y(x;\lambda)=\sum_{n=0}^\infty a_n(\lambda)x^n,\;a_0(\lambda)=1 \label{series}
					\end{equation}
				near $x=0$. Since the finite regular singular points of Equation \eqref{heun spectral eqn} are $x=0,1,\frac{8}{5}$, $y(\cdot;\lambda)$ fails to be analytic at $x=1$ precisely when the radius of convergence of $\eqref{series}$ is equal to one. To that end, we derive a recurrence relation for the coefficients given by
					\begin{equation}
						a_{n+2}(\lambda)=A_n(\lambda)a_{n+1}(\lambda)+B_n(\lambda)a_n(\lambda) \label{recurrence}
					\end{equation}
				where
					$$
						A_n(\lambda)=\frac{3 \lambda ^2+114 \lambda +52 n^2+32 \lambda  n+348 n+400}{16 (n+2) (2 n+13)}
					$$
				and
					$$
						B_n(\lambda)=-\frac{5 (\lambda +2 n+2) (\lambda +2 n+8)}{16 (n+2) (2 n+13)}
					$$
				with $a_{-1}(\lambda)=0$. For $n\in\mathbb N_0$, we define 
					$$
						r_n(\lambda):=\frac{a_{n+1}(\lambda)}{a_n(\lambda)}.
					$$
				Since $\lim_{n\to\infty}A_n(\lambda)=\frac{13}{8}$ and $\lim_{n\to\infty}B_n(\lambda)=-\frac{5}{8}$, the so-called characteristic equation of Equation \eqref{recurrence} is
					$$
						t^2-\frac{13}{8}t+\frac{5}{8}=0
					$$
				which has solutions $t_1=\frac{5}{8}$ and $t_2=1$. Poincar\'e's theorem for difference equations, see \cite{E05} or \cite{GS21} Appendix A, implies that either $a_n(\lambda)$ is zero eventually in $n$ or 
					\begin{equation}
						\lim_{n\to\infty}r_n(\lambda)=\frac{5}{8} \label{antigoal}
					\end{equation}
				or
					\begin{equation}
						\lim_{n\to\infty}r_n(\lambda)=1. \label{goal}
					\end{equation}
				We aim to prove that Equation \eqref{goal} holds true. 
				
				First, observe that $a_n(\lambda)$ cannot eventually be zero since, otherwise, backwards substitution would allow us to conclude that $a_0(\lambda)=0$ which is in clear contradiction with $a_0(\lambda)=1$. To rule out Equation \eqref{antigoal}, we first derive a recurrence relation for $r_n(\lambda)$ given by
					\begin{equation}
						r_{n+1}(\lambda)=A_n(\lambda)+\frac{B_n(\lambda)}{r_n(\lambda)} \label{r recurrence}
					\end{equation}
				with initial condition
					$$
						r_0(\lambda)=\frac{a_1(\lambda)}{a_0(\lambda)}=A_{-1}(\lambda)=\frac{1}{176} (\lambda +26) (3 \lambda +4).
					$$
				Furthermore, we define an approximate solution of Equation \eqref{r recurrence} by
					$$
						\tilde r_n(\lambda):=\lambda ^2 \left(\frac{3}{16 (n+1) (2 n+11)}+\frac{9}{4000 n^2}\right)+\lambda 
   \left(\frac{16 n+41}{8 (n+1) (2 n+11)}-\frac{1}{13 n}\right)+\frac{4 n+19}{4 n+22} \label{quasisolution}
					$$
				for $n\in\mathbb N$ which we call a \textit{quasisolution}. This quasisolution is intended to mimic the behavior of the actual solution $r_n(\lambda)$ for large $n$. We note that this quasisolution is not the \textit{canonical} quasisolution one would consider following the methods in earlier works. We will discuss this point in detail after the conclusion of this proof. Observe that for fixed $\lambda\in\overline{\mathbb H}$, $\lim_{n\to\infty}\tilde r_n(\lambda)=1$. If indeed $r_n(\lambda)$ remains close to the quasisolution, then we can exclude Equation \eqref{antigoal} implying that Equation \eqref{goal} must hold. To prove this, we define
					$$
						\delta_n(\lambda):=\frac{r_n(\lambda)}{\tilde r_n(\lambda)}-1 \label{delta relation}
					$$
				to measure the difference between $r_n(\lambda)$ and the quasisolution and derive a recurrence relation for this difference given by
					$$
						\delta_{n+1}(\lambda)=\varepsilon_n(\lambda)-C_n(\lambda)\frac{\delta_n(\lambda)}{1+\delta_n(\lambda)}
					$$
				where
					$$
						\varepsilon_n(\lambda)=\frac{A_n(\lambda)\tilde r_n(\lambda)+B_n(\lambda)}{\tilde r_n(\lambda)\tilde r_{n+1}(\lambda)}-1
					$$
				and
					\begin{equation}
						C_n(\lambda)=\frac{B_n(\lambda)}{\tilde r_n(\lambda)\tilde r_{n+1}(\lambda)}. \label{Cn}
					\end{equation}
			
				For $n\geq5$, we have the following estimates
					\begin{align*}
				|\delta_5(\lambda)|\leq\frac{1}{4}, \quad 
						|\varepsilon_n(\lambda)|\leq\frac{64+5n}{120(4+n)}, \quad 
					|C_n(\lambda)|\leq\frac{56+25n}{40(4+n)}.
					\end{align*}
				We will prove the third estimate while the first and second are established analogously. First, we bring $C_n(\lambda)$ into the form of a rational function, namely $C_n(\lambda)=\frac{P_1(n,\lambda)}{P_2(n,\lambda)}$ for polynomials $P_1, P_2\in\mathbb Z[n,\lambda]$. Explicit expressions are provided in Appendix \ref{explicit}. We can prove the estimate by first establishing it on the imaginary line and then extending it to all of $\overline{\mathbb H}$. This extension can be achieved by showing that $C_n(\lambda)$ is analytic and polynomially bounded on $\overline{\mathbb H}$ at which point the Phragm\'en-Lindel\"of principle achieves the desired extension. 
				
				Observe that for $t\in\mathbb R$, The inequality 
				\[|C_n(it)|\leq\frac{56+25n}{40(4+n)}\] is equivalent to the inequality 
				\[(40(4+n))^2|P_1(n,it)|^2-(56+25n)^2|P_2(n,it)|^2\leq0.\] For $t\in\mathbb R$ and $n\geq5$, a direct calculation shows that the coefficients of 
				\[(40(4+n))^2|P_1(n,it)|^2-(56+25n)^2|P_2(n,it)|^2\] are manifestly negative which establishes the desired estimate on the imaginary line. Now, we aim to extend the estimate to all of $\overline{\mathbb{H}}$. As $C_n(\lambda)$ is a rational function of polynomials in $\mathbb Z[n,\lambda]$, it is polynomially bounded. Furthermore, a direct calculation of the zeros of $P_2(n,\lambda)$ shows that they are contained in $\mathbb C\setminus\overline{\mathbb H}$ implying the analyticity of $C_n(\lambda)$ in $\overline{\mathbb{H}}$. Thus, the Phragm\'en-Lindel\"of principle extends the estimate to all of $\overline{\mathbb{H}}$. 
				
				By an inductive argument, we establish
					$$
						|\delta_n(\lambda)|\leq\frac{1}{4}
					$$
				for all $n\geq5$. Now, suppose Equation \eqref{antigoal} holds true. Then
					$$
						\frac{1}{4}\geq\lim_{n\to\infty}|\delta_n(\lambda)|=\lim_{n\to\infty}\bigg|\frac{r_n(\lambda)}{\tilde r_n(\lambda)}-1\bigg|=\frac{3}{8}
					$$ 
				which is a clear contradiction. Thus, it must be the case that Equation \eqref{goal} holds. Thus, Equation \eqref{heun spectral eqn} does not have solutions which are analytic at one.  Consequently, we can exclude analytic solution of Equation \eqref{susy spectral eqn}, which contradicts our assumption.
\end{proof}
		
			\begin{remark} \label{spectral difficulty}
				A natural first guess for a quasisolution would be
					$$
						\tilde r_n(\lambda)=\lambda ^2 \left(\frac{3}{16 (n+1) (2 n+11)}\right)+\lambda 
   \left(\frac{16 n+41}{8 (n+1) (2 n+11)}\right)+\frac{4 n+19}{4 n+22}
					$$
				following the methods in \cite{CDGH15}, \cite{G18}, and \cite{GS21}. The quadratic and linear terms in $\lambda$ come from studying the large $|\lambda|$ behavior of $A_n(\lambda)$ while the constant in $\lambda$ term comes from fitting the first few iterates of $r_n(\lambda)$ for small $|\lambda|$. However, it appears that this quasisolution does not work when trying to obtain any reasonable estimates on $\delta_5(\lambda)$, $\varepsilon_n(\lambda)$, and $C_n(\lambda)$. Lower-order corrections to the linear and quadratic terms, to the best of our knowledge, appear to be essential in obtaining such estimates. This may be important for solving future spectral problems with this method.
			\end{remark}

		\subsection{Decay of the Linearized Flow} \label{Control of the Linearized Flow}
			Lemma \ref{spectrum} shows that $1,4\in\sigma_p(\mathbf L)$ are isolated. This allows us to define the following Riesz projections.
				
				\begin{definition} \label{Riesz projection defn}
					Let $\gamma_1:[0,2\pi]\to\mathbb C$ and $\gamma_4:[0,2\pi]\to\mathbb C$ be defined by $\gamma_1(t)=1+\frac{1}{2}e^{it}$ and $\gamma_4(t)=4+\frac{1}{2}e^{it}$. Then we set
						$$
							\mathbf P_j:=\frac{1}{2\pi i}\int_{\gamma_j}\mathbf R_\mathbf L(\lambda)d\lambda,\;j=1,4.
						$$
				\end{definition}
		
				\begin{proposition} \label{projection}
					The operators $\mathbf P_j\in\mathcal B(\mathcal H_R)$, $j=1,4$, commute with the semigroup $\big(\mathbf S(s)\big)_{s\geq0}$ and are mutually transversal, i.e., 
						$$
							\mathbf P_1\mathbf P_4=\mathbf P_4\mathbf P_1=\mathbf 0.
						$$
					Furthermore, we have
						$$
							\range\mathbf P_j=\langle\mathbf f_j^*\rangle
						$$
					and
						$$
							\mathbf S(s)\mathbf P_j\mathbf f=e^{js}\mathbf P_j\mathbf f,\;s\geq0,\;\mathbf f\in\mathcal H_R,\;j=1,4.
						$$
				\end{proposition}
				\begin{proof}
					Boundedness, transversality, and commuting with semigroup follow from abstract theory, see \cite{K95} and \cite{EN00}. In the following, we handle both cases $j=1$ and $j=4$ simultaneously until the very end at which point the arguments slightly diverge.
					
					We aim to show $\range\mathbf P_j=\langle\mathbf f_j^*\rangle$. The inclusion $\langle\mathbf f_j^*\rangle\subseteq\range\mathbf P_j$ follows from abstract theory, see \cite{K95}. For the reverse inclusion, observe that $\mathbf P_j$ decomposes $\mathcal H_R$ as $\mathcal H_R=\range\mathbf P_j\oplus\range(\mathbf I-\mathbf P_j)$ and the operator $\mathbf L$ decomposes into the parts $\mathbf L|_{\range\mathbf P_j}$ and $\mathbf L|_{\range(\mathbf I-\mathbf P_j)}$ acting on $\range\mathbf P_j$ and $\range(\mathbf I-\mathbf P_j)$ respectively. The spectra of these operators are
						$$
							\sigma(\mathbf L|_{\range\mathbf P_j})=\{j\},\;\sigma(\mathbf L|_{\range(\mathbf I-\mathbf P_j)})=\sigma(\mathbf L)\setminus\{j\}.
						$$
					We claim that $\range\mathbf P_j$ is finite-dimensional. To see this, suppose that $\dim\range\mathbf P_j=\infty$. Then, Theorem 5.28 of \cite{K95} implies that $j\in\sigma_e(\mathbf L)$. Since $\mathbf L'$ is compact and the essential spectrum is stable under compact perturbations, we also have that $j\in\sigma_e(\mathbf L-\mathbf L')$. This is clearly a contradiction since $\mathbf L-\mathbf L'=\mathbf L_7-2\mathbf I$ and $\sigma(\mathbf L_7-2\mathbf I)\subseteq\{z\in\mathbb C:\Re z\leq-\frac{3}{2}\}$. 
					
					Thus, the part $\mathbf L|_{\range\mathbf P_j}$ acts on a finite-dimensional Hilbert space with $\sigma(\mathbf L|_{\range\mathbf P_j})=\{j\}$. Consequently, $j\mathbf I-\mathbf L|_{\range\mathbf P}$ is nilpotent since $0$ is its only spectral point and is an eigenvalue. So, there exists a minimal $\ell_j\in\mathbb N$ with $(j\mathbf I-\mathbf L|_{\range\mathbf P_j})^{\ell_j}\mathbf f=\mathbf 0$ for all $\mathbf f\in\range\mathbf P_j$. If $\ell_j=1$, then the reverse inclusion follows. 
					
					Suppose $\ell_j\neq1$. Then there exists a nonzero $\mathbf f_j=(f_{j,1},f_{j,2})\in\range\mathbf P_j\subset H^6_\text{rad}(\mathbb B^7_R)\times H^5_\text{rad}(\mathbb B^7_R)\subset C^2(0,R)\times C^1(0,R)$ such that $\mathbf f_j\in\ker(j\mathbf I-\mathbf L|_{\range\mathbf P_j})\subseteq\ker(j\mathbf I-\mathbf L)$. By Lemma \ref{spectrum}, we have $\ker(j\mathbf I-\mathbf L)=\langle\mathbf f_j^*\rangle$. Thus, $\mathbf f_j$ solves the equation
						$$
							\alpha \mathbf f_j^*=(j\mathbf I-\mathbf L)\mathbf f_j.
						$$
					for some $\alpha\in\mathbb C\setminus\{0\}$. Without loss of generality, we take $\alpha=1$. Consequently, the first component of  $\mathbf f_j$ solves the ODE
						\begin{equation}
							 f_{j,1}''(y)+p_j(y)f_{j,1}'(y)+q_j(y)f_{j,1}(y)=\frac{G_j(y)}{c_{12}(y)} \label{inhom spectral eqn}
						\end{equation}
					where
						$$
							p_j(y)=\frac{c_{11}(y)+(j+2)c_{21}(y)}{c_{12}(y)},
						$$
						$$
							q_j(y)=\frac{(j+2)(c_{20}(y)-j-2)+V(y)}{c_{12}(y)},
						$$
					and
						$$
							G_j(y)=c_{21}(y) [f^*_{j,1}]'(y)+(c_{20}(y)-4-2j)f_{j,1}^*(y). \label{Gj}
						$$
					Let $I_j(y)$ be an antiderivative of $p_j(y)$. For instance, the explicit functions
						$$
							I_1(y)=\log\Bigg(\frac{y^6(1-4 y^2)\sqrt{y^2+2} \sqrt{y^2+2\sqrt{y^2+2}+3}}{(y^2+1) \sqrt{3 \sqrt{y^2+2}-y+4} \sqrt{3 \sqrt{y^2+2}+y+4}}\Bigg)
						$$
					and
						$$
							I_4(y)=\log\Bigg(\frac{y^6 (1-4y^2)^4 \sqrt{y^2+2} \sqrt{y^2+2 \sqrt{y^2+2}+3}}{(y^2+1)\big(8 y^2+24 \sqrt{y^2+2}+34\big)^2}\Bigg)
						$$
					suffice. We obtain a fundamental system for the homogeneous equation given by
						\begin{align*}
							&\phi_j(y):=f_{j,1}^*(y)
							\\
							&\psi_j(y):=f_{j,1}^*(y)\int_{\frac{1}{4}}^{y}\exp\big(-I_j(y')\big)f_{j,1}^*(y)^{-2}dy'
						\end{align*}
					where the lower bound of integration in $\psi_j$ is chosen arbitrarily. Furthermore, observe that 
						$$
							\exp\big(-I_j(y)\big)\simeq y^{-6}\Big(\frac{1}{2}-y\Big)^{-j}
						$$
					which implies that for the second solution we have
						$$
							|\psi_1(y)|\simeq y^{-5}\bigg|\log\Big(\frac{1}{2}-y\Big)\bigg|,\;\;\;|\psi_1'(y)|\simeq y^{-6}\Big(\frac{1}{2}-y\Big)^{-1}
						$$
					and
						$$
							|\psi_4(y)|\simeq y^{-5}\Big(\frac{1}{2}-y\Big)^{-3},\;\;\;|\psi_4'(y)|\simeq y^{-6}\Big(\frac{1}{2}-y\Big)^{-4}.
						$$
					Observe that the Wronskian is precisely $\exp\big(-I_j(y')\big)$ up to some constant multiple. This implies that we have
						$$
							|W(\phi_j,\psi_j)(y)|\simeq y^{-6}\Big(\frac{1}{2}-y\Big)^{-j}.
						$$
					Variation of parameters shows that $f_{j,1}$ must be of the form
						\begin{align*}
							f_{j,1}(y)=&c_1^{(j)}\phi_j(y)+c_2^{(j)}\psi_j(y)
							\\
							&-\phi_j(y)\int_0^y\frac{\psi_j(\rho)}{W(\phi_j,\psi_j)(\rho)}\frac{G_j(\rho)}{c_{12}(\rho)}d\rho+\psi_j(y)\int_0^y\frac{\phi_j(\rho)}{W(\phi_j,\psi_j)(\rho)}\frac{G_j(\rho)}{c_{12}(\rho)}d\rho
						\end{align*}
					for $y\in(0,\frac{1}{2})$. Taking the limit $y\to0^+$ yields $c_2^{(j)}=0$. Thus, we are left with 
						$$
							f_{j,1}(y)=c_1^{(j)}\phi_j(y)-\phi_j(y)\int_0^y\frac{\psi_j(\rho)}{W(\phi_j,\psi_j)(\rho)}\frac{G_j(\rho)}{c_{12}(\rho)}d\rho+\psi_j(y)\int_0^y\frac{\phi_j(\rho)}{W(\phi_j,\psi_j)(\rho)}\frac{G_j(\rho)}{c_{12}(\rho)}d\rho
						$$
					Based on the above asymptotics we find
						$$
							\lim_{y\to\frac{1}{2}^-}\int_0^y\frac{\psi_j(\rho)}{W(\phi_j,\psi_j)(\rho)}\frac{G_j(\rho)}{c_{12}(\rho)}d\rho
						$$
					exists. As a consequence, in order to control the third term near $y=\frac{1}{2}$, we must have
						$$
							\int_0^{\frac{1}{2}}\frac{\phi_j(\rho)}{W(\phi_j,\psi_j)(\rho)}\frac{G_j(\rho)}{c_{12}(\rho)}d\rho=0
						$$
					For $j=4$, the integrand has a definite sign. Thus, the above integral vanishing yields a contradiction. For $j=1$, the integral can be computed explicitly and is nonzero which again yields a contradiction. Thus, we must have $(j\mathbf I-\mathbf L)\mathbf f=\mathbf 0$ for all $\mathbf f\in\range\mathbf P_j$ which, with Proposition \ref{spectrum}, implies $\range\mathbf P_j\subseteq\langle\mathbf f_j^*\rangle$. 
					
					Lastly, the claim 
						$$
							\mathbf S(s)\mathbf P_j\mathbf f=e^{js}\mathbf P_j\mathbf f,\;s\geq0,\;\mathbf f\in\mathcal H_R,\;j=1,4
						$$
					is a direct consequence of $\range\mathbf P_j=\langle\mathbf f_j^*\rangle$
				\end{proof}
		
			We now state and prove the main result on the linearized equation.
				\begin{theorem} \label{linear stability}
					Let $\mathbf P:=\mathbf P_1+\mathbf P_4$. Then there exist $\omega_0>0$ and $C\geq1$ such that 
						$$
							\|\mathbf S(s)(\mathbf I-\mathbf P)\mathbf f\|_{\mathcal H_R}\leq Ce^{-\omega_0s}\|(\mathbf I-\mathbf P)\mathbf f\|_{\mathcal H_R}
						$$
					for all $s\geq0$ and all $\mathbf f\in\mathcal H_R$.
				\end{theorem}
				\begin{proof}
					Lemma \ref{restrict to point spectrum} and Proposition \ref{spectrum} imply that $\sup\{\Re\lambda:\lambda\in\mathbb C\setminus S_\epsilon\}<0$. Thus, by Theorem B.1.iii of \cite{G21} we obtain the desired result. 
				\end{proof}
		
	\section{Nonlinear Stability Analysis} \label{Nonlinear Stability Analysis}
		\subsection{Well-Posedness and Decay of the Nonlinear Evolution} \label{Well-Posedness and Decay of the Nonlinear Evolution}
			We now turn our attention to the nonlinear problem
				\begin{equation}
				\begin{cases}
					\partial_s\Phi(s)=\mathbf L\Phi(s)+\mathbf N\big(\Phi(s)\big)
					\\
					\Phi(0)=\Phi_0 \label{NL problem}
				\end{cases}
				\end{equation}
			for initial data $\Phi_0$ contained in a small ball in $\mathcal H_R$. Using with the semigroup $\big(\mathbf S(s)\big)_{s\geq0}$ we appeal to Duhamel's formula and reformulate Equation \eqref{NL problem} as the integral equation
				\begin{equation}
					\Phi(s)=\mathbf S(s)\Phi_0+\int_{0}^s\mathbf S(s-s')\mathbf N\big(\Phi(s')\big)ds'. \label{Duhamel}
				\end{equation}
			As a first step, we prove a mapping property and local Lipschitz bound on the nonlinearity.
	
			\begin{lemma} \label{nonlinear bound}
				We have $\mathbf N:\mathcal H_R\to\mathcal H_R$ and satisfies the bound
					$$
						\|\mathbf N(\mathbf f)-\mathbf N(\mathbf g)\|_{\mathcal H_R}\lesssim\big(\|\mathbf f\|_{\mathcal H_R}+\|\mathbf g\|_{\mathcal H_R}\big)\|\mathbf f-\mathbf g\|_{\mathcal H_R}
					$$
				for all $\mathbf f,\mathbf g\in\mathcal H_R$.
			\end{lemma}
			\begin{proof}
				Recalling the definition of $\mathbf N$, we find
					\begin{align*}
						\|\mathbf N(\mathbf f)-\mathbf N(\mathbf g)\|_{\mathcal H_R}&=\Big\|N\big(\cdot,f_1(\cdot)\big)-N\big(\cdot,g_1(\cdot)\big)\Big\|_{H_\text{rad}^5(\mathbb B^7_R)}
						\\
						&\lesssim\big\|f_1^2-g_1^2\big\|_{H_\text{rad}^5(\mathbb B^7_R)}
						\\
						&\lesssim\|f_1+g_1\|_{H_\text{rad}^5(\mathbb B^7)}\|f_1-g_1\|_{H_\text{rad}^5(\mathbb B^7_R)}
						\\
						&\lesssim\big(\|\mathbf f\|_{\mathcal H_R}+\|\mathbf g\|_{\mathcal H_R}\big)\|\mathbf f-\mathbf g\|_{\mathcal H_R}
					\end{align*}
				where the second to third line follows from the Banach algebra property of $H_\text{rad}^5(\mathbb B^7_R)$. The claim $\mathbf N:\mathcal H_R\to\mathcal H_R$ follows from $\mathbf N(\mathbf 0)=\mathbf 0$.
			\end{proof}

		Due to the instabilities associated with the eigenvalues $\lambda=1,4$, Equation \eqref{Duhamel} will not, in general, have global solutions that decay. Instead, we consider a modified equation which allows us to correct for these instabilities and achieve global existence and decay. Upon reconnecting to the problem in physical coordinates, we will in fact show that for arbitrary, small perturbations of $u_1^*$, there is a way to adjust this perturbation and a choice of $T$ close to $1$ for which this modification vanishes and the corresponding solution converges to $u_T^*$.

			\begin{definition}
				For $\omega_0$ as in Theorem \ref{linear stability}, we define the Banach space
					$$
						\mathcal X_R:=\{\Phi\in C([0,\infty),\mathcal H_R):\|\Phi\|_{\mathcal X_R}<\infty\}
					$$
				where
					$$
						\|\Phi\|_{\mathcal X_R}:=\sup_{s>0}\Big(e^{\omega_0s}\|\Phi(s)\|_{\mathcal H_R}\Big).
					$$
				Furthermore, we define $\mathbf C_j:\mathcal X_R\times\mathcal H_R\to\range\mathbf P_j$, $j=1,4$ by
					$$
						\mathbf C_j(\Phi,\mathbf f):=\mathbf P_j\bigg(\mathbf f+\int_{0}^\infty e^{-js'}\mathbf N\big(\Phi(s')\big)ds'\bigg)
					$$
				and set $\mathbf C:=\mathbf C_1+\mathbf C_4$.
			\end{definition}

			With this, we study the modified equation
				\begin{equation}
					\Phi(s)=\mathbf S(s)\big[\mathbf f-\mathbf C(\Phi,\mathbf f)\big]+\int_{0}^s\mathbf S(s-s')\mathbf N\big(\Phi(s')\big)ds'. \label{modified Duhamel}
				\end{equation}
			For Equation \eqref{modified Duhamel}, we show that for all sufficiently small data $\mathbf f$, there exists a unique solution in the space $\mathcal X_R$ depending Lipschitz continuously on $\mathbf f$. In other words, the nonlinear problem is globally well-posed for all sufficiently small initial data and the corresponding solutions decay exponentially as $s\to\infty$.
			
			\begin{proposition} \label{mod wp}
				For all sufficiently large $c>0$ and sufficiently small $\delta>0$ and any $\mathbf f\in\mathcal H_R$ satisfying $\|\mathbf f\|_{\mathcal H_R}\leq\frac{\delta}{c}$, there exists a unique solution $\Phi_{\mathbf f}\in C([0,\infty),\mathcal H_R)$ of Equation \eqref{modified Duhamel} that satisfies $\|\Phi_{\mathbf f}(s)\|_{\mathcal H_R}\leq\delta e^{-\omega_0s}$ for all $s\geq0$. Furthermore, the solution map $\mathbf f\mapsto\Phi_{\mathbf f}$ is Lipschitz as a function from a small ball in $\mathcal H_R$ to $\mathcal X_R$.
			\end{proposition}
			\begin{proof}
				Set 
					$$
						\mathcal Y_\delta:=\{\Phi\in\mathcal X_R:\|\Phi\|_{\mathcal X_R}\leq\delta\}
					$$ 
				and define the map
					$$
						\mathbf K_{\mathbf f}(\Phi)(s):=\mathbf S(s)\big[\mathbf f-\mathbf C(\Phi,\mathbf f)\big]+\int_{0}^s\mathbf S(s-s')\mathbf N\big(\Phi(s')\big)ds'.
					$$
				We aim to show that $\mathbf K_{\mathbf f}:\mathcal Y_\delta\to\mathcal Y_\delta$ and is a contraction. 
				
				First, observe that by Theorem \ref{linear stability} and Proposition \ref{projection} we obtain
					$$
						\mathbf P_j\mathbf K_{\mathbf f}(\Phi)(s)=-\int_s^\infty e^{j(s-s')}\mathbf P_j\mathbf N\big(\Phi(s')\big)ds'.
					$$
				From Lemma \ref{nonlinear bound} and the fact that $\mathbf N(\mathbf0)=\mathbf0$, we have the estimate
					\begin{align*}
						\|\mathbf P_j\mathbf K_{\mathbf f}(\Phi)(s)\|_{\mathcal H_R}&\lesssim e^{js}\int_s^\infty e^{-js'}\|\Phi(s')\|_{\mathcal H_R}^2ds'
						\\
						&\lesssim e^{js}\|\Phi\|_{\mathcal X_R}^2\int_s^\infty e^{-js'-2\omega_0s'}ds' \lesssim \delta^2e^{-2\omega_0s}.
					\end{align*}
				By Proposition \ref{projection}, we have $(\mathbf I-\mathbf P)\mathbf C(\Phi,\mathbf f)=\mathbf 0$ which implies
					$$
						(\mathbf I-\mathbf P)\mathbf K_{\mathbf f}(\Phi)(s)=\mathbf S(s)(\mathbf I-\mathbf P)\mathbf f+\int_{0}^s\mathbf S(s-s')(\mathbf I-\mathbf P)\mathbf N\big(\Phi(s')\big)ds'.
					$$
				By Theorem \ref{linear stability}, we obtain
					\begin{align*}
						\|(\mathbf I-\mathbf P)\mathbf K_{\mathbf f}(\Phi)(s)\|_{\mathcal H_R}&\lesssim\frac{\delta}{c}\|\mathbf f\|_{\mathcal H_R}+\int_{0}^se^{-\omega_0(s-s')}\|\mathbf N\big(\Phi(s')\big)\|_{\mathcal H_R}ds'
						\\
						&\lesssim\frac{\delta}{c}e^{-\omega_0s}+e^{-\omega_0s}\int_{0}^se^{\omega_0s'}\|\Phi(s')\|_{\mathcal H_R}^2ds'
						\\
						&\lesssim\frac{\delta}{c}e^{-\omega_0s}+\|\Phi\|_{\mathcal X_R}^2e^{-\omega_0s}\int_{0}^se^{-\omega_0s'}ds'
						\\
						&\lesssim\frac{\delta}{c}e^{-\omega_0s}+\delta^2e^{-\omega_0s}
					\end{align*}
				for all $s\geq0$. Thus, for all sufficiently large $c$ and sufficiently small $\delta$, we can ensure
					$$
						\|\mathbf K_{\mathbf f}(\Phi)(s)\|_{\mathcal H_R}\leq\delta e^{-\omega_0s}.
					$$
				Consequently, we see that $\mathbf K_{\mathbf f}:\mathcal Y_\delta\to\mathcal Y_\delta$. 
				
				We claim that $\mathbf K_{\mathbf f}$ is a contraction map. Given $\Phi,\Psi\in\mathcal Y_\delta$,
					$$
						\mathbf P_j\mathbf K_{\mathbf f}(\Phi)(s)-\mathbf P_j\mathbf K_{\mathbf f}(\Psi)(s)=-\int_s^\infty e^{j(s-s')}\mathbf P_j\Big(\mathbf N\big(\Phi(s')\big)-\mathbf N\big(\Psi(s')\big)\Big)ds'.
					$$
				By Lemma \ref{nonlinear bound}
					\begin{align*}
						\|\mathbf P_j\mathbf K_{\mathbf f}(\Phi)(s)&-\mathbf P_j\mathbf K_{\mathbf f}(\Psi)(s)\|_{\mathcal H_R}
						\\
						&\lesssim e^{js}\int_s^\infty e^{-js'}\big(\|\Phi(s')\|_{\mathcal H_R}+\|\Psi(s')\|_{\mathcal H_R}\big)\|\Phi(s')-\Psi(s')\|_{\mathcal H_R}ds'
						\\
						&\lesssim\delta\|\Phi-\Psi\|_{\mathcal X_R}e^{js}\int_s^\infty e^{-js'-2\omega_0s'}ds'
						\lesssim\delta e^{-2\omega_0s}\|\Phi-\Psi\|_{\mathcal X_R}.
					\end{align*}
				Furthermore,
					$$
						(\mathbf I-\mathbf P)\mathbf K_{\mathbf f}(\Phi)(s)-(\mathbf I-\mathbf P)\mathbf K_{\mathbf f}(\Psi)(s)=\int_0^s\mathbf S(s-s')(\mathbf I-\mathbf P)\Big(\mathbf N\big(\Phi(s')\big)-\mathbf N\big(\Psi(s')\big)\Big)ds'.
					$$
				By Theorem \ref{linear stability} and Lemma \ref{nonlinear bound}, we obtain
					\begin{align*}
						\|(\mathbf I-\mathbf P)\mathbf K_{\mathbf f}(\Phi)(s)&-(\mathbf I-\mathbf P)\mathbf K_{\mathbf f}(\Psi)(s)\|_{\mathcal H_R}
						\\
						&\lesssim\int_0^s e^{-\omega_0(s-s')}\big(\|\Phi(s')\|_{\mathcal H_R}+\|\Psi(s')\|_{\mathcal H_R}\big)\|\Phi(s')-\Psi(s')\|_{\mathcal H_R}ds'
						\\
						&\lesssim\delta\|\Phi-\Psi\|_{\mathcal X_R}e^{-\omega_0s}\int_0^s e^{-\omega_0s'}ds'
						\lesssim\delta e^{-\omega_0s}\|\Phi-\Psi\|_{\mathcal X_R}.
					\end{align*}
				Thus,
					$$
						\|\mathbf K_{\mathbf f}(\Phi)-\mathbf K_{\mathbf f}(\Psi)\|_{\mathcal X_R}\lesssim\delta\|\Phi-\Psi\|_{\mathcal X_R}
					$$
				and by considering smaller $\delta$ if necessary, we see that $\mathbf K_{\mathbf f}$ is a contraction on $\mathcal Y_\delta$. The Banach fixed point theorem implies the existence of a unique fixed point $\Phi_\mathbf f\in\mathcal Y_\delta$ of $\mathbf K_\mathbf f$. 
			
				We now claim that the solution map $\mathbf f\mapsto\Phi_\mathbf f$ is Lipschitz. Observe that
					\begin{align*}
						\|\Phi_\mathbf f-\Phi_\mathbf g\|_{\mathcal X_R}&=\|\mathbf K_\mathbf f(\Phi_\mathbf f)-\mathbf K_\mathbf g(\Phi_\mathbf g)\|_{\mathcal X_R}
						\\
						&\leq\|\mathbf K_\mathbf f(\Phi_\mathbf f)-\mathbf K_\mathbf f(\Phi_\mathbf g)\|_{\mathcal X_R}+\|\mathbf K_\mathbf f(\Phi_\mathbf g)-\mathbf K_\mathbf g(\Phi_\mathbf g)\|_{\mathcal X_R}
						\\
						&\lesssim\delta\|\Phi_\mathbf f-\Phi_\mathbf g\|_{\mathcal X_R}+\|\mathbf K_\mathbf f(\Phi_\mathbf g)-\mathbf K_\mathbf g(\Phi_\mathbf g)\|_{\mathcal X_R}.
					\end{align*}
				A direct calculation shows
					$$
						\mathbf K_\mathbf f(\Phi_\mathbf g)(s)-\mathbf K_\mathbf g(\Phi_\mathbf g)(s)=\mathbf S(s)(\mathbf I-\mathbf P)(\mathbf f-\mathbf g).
					$$
				Theorem \ref{linear stability} yields
					$$
						\|\mathbf K_\mathbf f(\Phi_\mathbf g)(s)-\mathbf K_\mathbf g(\Phi_\mathbf g)(s)\|_{\mathcal H_R}\lesssim e^{-\omega_0s}\|\mathbf f-\mathbf g\|_{\mathcal H_R}.
					$$
				Thus, we have
					$$
						\|\Phi_\mathbf f-\Phi_\mathbf g\|_{\mathcal X_R}\lesssim\delta\|\Phi_\mathbf f-\Phi_\mathbf g\|_{\mathcal X_R}+\|\mathbf f-\mathbf g\|_{\mathcal H_R}.
					$$
				Again, considering smaller $\delta$ if necessary yields the result.
			\end{proof}
		
	\section{Preparation of Hyperboloidal Initial Data}\label{Preparation of Hyperboloidal Initial Data}
		In this section, we construct the functions $(Y_1,Y_2)$ mentioned in the statement of Theorem \ref{stability on whole space}. As this is nontrivial, we begin by first motivating and outlining the main strategy. 
		 
		\subsection{Motivation and Overview for the Construction of $(Y_1,Y_2)$} \label{Motivation and Overview for the Construction of (Y_1,Y_2)}
			Our goal is to evolve data of the form $u_1^*[0]+(f,g)$ for sufficiently small, smooth, compactly supported, radial functions $f,g$ into the region $\Omega_{T,R}$ according to Equation \eqref{qwe}. As this region is not foliated by surfaces of constant physical time, this evolution must occur in two steps: first along hypersurfaces of constant physical time and then along hypersurfaces of constant hyperboloidal time. Carrying out this two-step evolution is rather nontrivial due to the fact that $\mathbf L$ has a genuine eigenvalue that does not come from a symmetry of the equation. We emphasize that this is a new difficulty compared to the previous works \cite{BDS19} and \cite{DO21}. In order to explain the nontrivial nature of this problem properly, we will first informally describe a natural approach one might naively attempt and then describe how we adapt this approach.
		
		 For the moment, let $T$ be some number close to $1$. The region $\Omega_{T,R}$ can be covered by a mix of hyperboloids and slices of constant physical time. With this in mind, a natural first step would be to solve the  quadratic wave equation, Equation \eqref{qwe}, with initial data $u[0]=u_1^*[0]+(f,g)$ for some short time using the standard Cauchy theory in physical coordinates. To continue the evolution to the rest of $\Omega_{T,R}$, one might expect to restrict this solution on some hyperboloid and evolve further using the nonlinear theory developed in Section \ref{Nonlinear Stability Analysis}. In fact, this is precisely what is done in \cite{BDS19} and \cite{DO21}. Of course, Equation \eqref{modified Duhamel} is not the quadratic wave equation due to the correction term. So, one might hope that there exists at least one choice of $T$ for which the correction term, $\mathbf C$, vanishes. If this were possible, then solutions of Equation \eqref{modified Duhamel} would in fact yield solutions of the quadratic wave equation in $\Omega_{T,R}$. An obstruction to this is that the correction term is a sum of two terms, one for each unstable eigenvalue. That is, $\mathbf C=\mathbf C_1+\mathbf C_4$ with $\mathbf C_1$ correcting for the eigenvalue $\lambda=1$ and $\mathbf C_4$ correcting for the eigenvalue $\lambda=4$. Without an additional parameter to vary, one cannot hope to guarantee the vanishing of both correction terms.
		
		Now, recall the solution $F_4^*$ of the quadratic wave equation linearized around $u_T^*$ given in Equation \eqref{Eq:Unst_Mode}. Translating to hyperboloidal similarity coordinates, we have the transformations $(F_4^*\circ\eta_T)(s,y)=e^{6s} f_{4,1}^*(y)$ and $\partial_s(F_4^*\circ\eta_T)(s,y)=e^{6s} f_{4,2}^*(y)$. The role of the correction term $\mathbf C_4$ is to remove the contribution of $\mathbf f_4^*$ from hyperboloidal initial data. Thus, it seems plausible to expect that for data of the form $u_1^*[0]+(f,g)+\alpha F_4^*[0]$, the correction term $\mathbf C_4$ might vanish for at least one choice of $\alpha$. As stated, it is not possible to guarantee this and continue the evolution of such data along hyperboloids using techniques as in \cite{BDS19} or \cite{DO21}. This is due to the fact that the physical evolution of such data cannot necessarily be contained in a single ball on a hyperboloid and, as a consequence, the nonlinear theory from Section \ref{Nonlinear Stability Analysis} cannot be applied in a meaningful way. 
		
		As a remedy, one might expect that data of the form $u_1^*[0]+(f,g)+\alpha\chi F_4^*[0]$, for some smooth cutoff function $\chi$ and some choice of $\alpha$ and $T$, might work. Though this may be possible, it appears extremely difficult to continue the evolution of such data along hyperboloids in a controllable way. The reason for this difficulty is that one proves that there are parameters $\alpha$ and $T$ for which $\mathbf C$ vanishes via a fixed point argument. In order to run this fixed point argument one needs two crucial pieces of information. On the one hand, we need to guarantee smallness and uniform control of the derivatives of the solution produced.    However, since $\chi F_4^*$ is not a true solution of the quadratic wave equation linearized around $u_1^*$, this cannot be guaranteed. Secondly, running the fixed point argument will eventually necessitate that the Riesz projection $\mathbf P_4$ applied to a portion of the data on an initial hyperboloid does not vanish. Proving this is difficult unless the data is of a rather explicit form. 
		
		With these requirements in mind, we can adapt the naive approach. Let's say that any sufficiently smooth perturbation of $u_1^*$ of unit size can be evolved using the standard Cauchy theory in physical coordinates for at least a length of time $t_0>0$. For technical reasons, we impose the condition that our perturbation have support contained in the interval $[0,r_0)$ with $r_0:=\frac{t_0}{4}$. We have two conditions which determine the proper replacement, denoted by $(Y_1,Y_2)$, for the term $\chi F_4^*[0]$ in our perturbation:

			\begin{enumerate}

				\item First, we require that $Y_1$ be the restriction of a solution $u_\text{lin}$ of the linearized equation \eqref{Eq:QWE_Linearized} at $t=0$ with $Y_2$ being its time derivative at $t=0$.
				
\item Furthermore, the restriction of $u_\text{lin}$ to a particular hyperboloidal time slice should be proportional to $\chi_{t_0}\mathbf f_4^*$ where $\chi_{t_0}$ is a specific smooth cutoff function with support determined by $t_0$. The support of $\chi_{t_0}$ is chosen precisely so that, when viewed in spacetime, its domain of influence at $t=0$ is contained within the interval $[0,r_0)$. Furthermore, the support is also chosen so that $\mathbf P_4$ applied to the previously mentioned portion of the solution does not vanish.
			\end{enumerate} 
			
		The second property guarantees the desired condition involving the Riesz projection while the first ensures the required smallness and uniform control on derivatives of the local solution obtained by solving the quadratic wave equation. The two properties are met by solving the linearized equation in two different ways; first in hyperboloidal similarity coordinates and second in physical coordinates. The condition on the support of the cutoff ensures that both ways of solving the linearized equation produce the same result in the overlapping region. Before carrying this out, we outline the construction and proof. 
		
		 Our goal will be to construct a smooth solution of Equation \eqref{Eq:QWE_Linearized}, with $T=1$, for $(t,r)\in\Lambda_{t_0}$ defined by
		 	
		 	\begin{align}\label{Def:Lambda_Region}
				\Lambda_{t_0}:=[-t_0,t_0]\times[0,\infty)\cup\{(t,r)\in\mathbb R\times[0,\infty):-r+r_0\leq t\leq r-r_0\}
			\end{align}
		satisfying the above two properties. To achieve this, one can first solve the abstract initial value problem
			$$
			\begin{cases}
				\partial_s\Phi(s)=\mathbf L\Phi(s)
				\\
				\Phi(s_0)=\chi_{t_0}\mathbf f_4^*
			\end{cases}
			$$
		for $s\geq s_0$ with
			$$
				s_0:=\log\Big(-\frac{2h(0)}{2+r_0}\Big)
			$$
		on the space $\mathcal H^k_{1/2}$ for any $k\in\mathbb N$. The number $s_0$ is chosen so that the hyperboloids $\eta_{1+\beta}(s_0,y)$, $y\geq0$, lie entirely within $\Lambda_{t_0}$ for $\beta$ sufficiently small. Furthermore, the cutoff $\chi_{t_0}$ is chosen to be non-increasing and to have support contained in the interval $[0,y_0)$ with $y_0$ to be defined later. This number is chosen precisely so that the domain of influence of $\supp(\chi_{t_0})$, when viewed in spacetime, at $t=0$ is contained in the interval $[0,r_0)$. As a consequence, one can prove that the solution is smooth and translates it to a smooth solution of the quadratic wave equation linearized around $u_1^*$ in the spacetime region $\eta_1([s_0,\infty)\times[0,\frac{1}{2}))$, see Figure \ref{eta1_region}. Let's call this solution $u_\text{lin}$.
			\begin{figure}
				\includegraphics[scale=0.75]{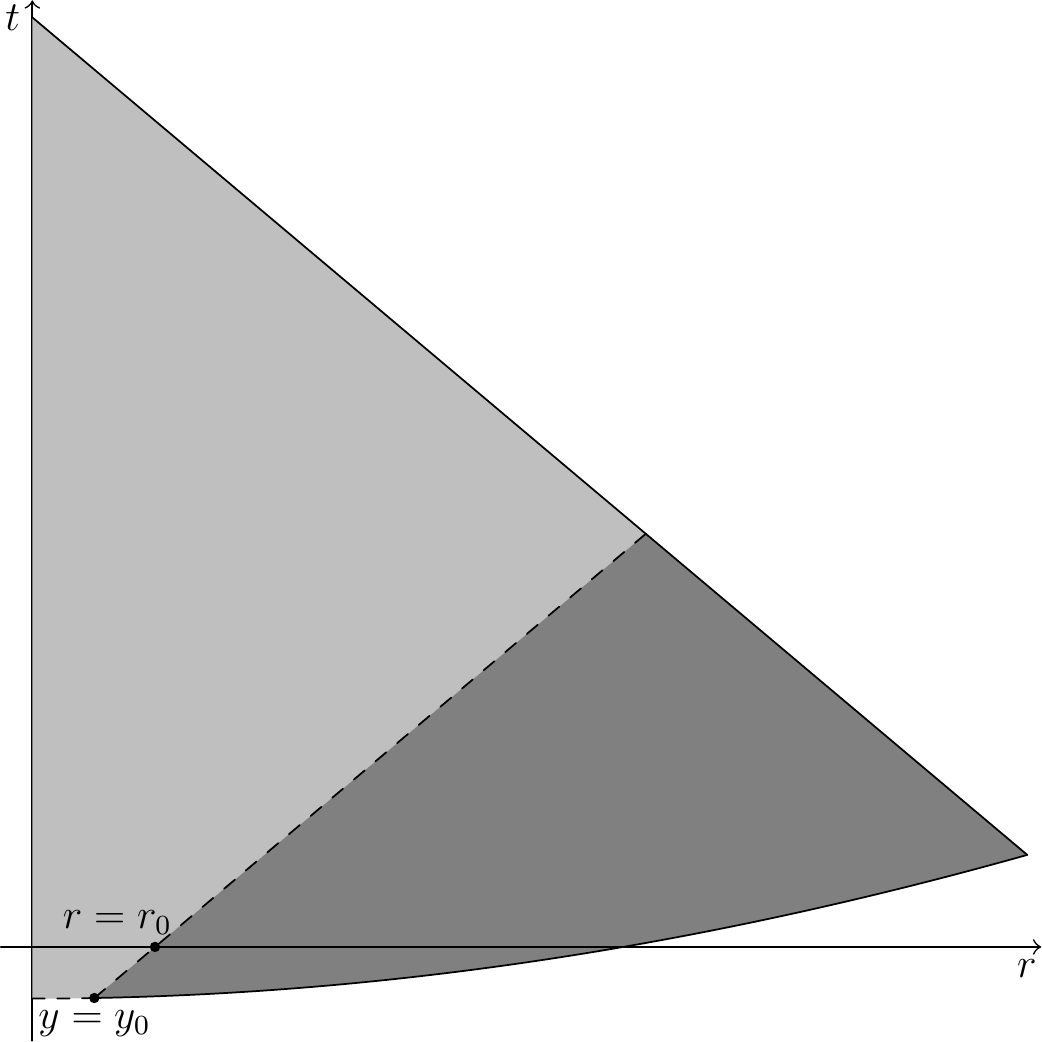}
				\centering
				\caption{A spacetime diagram depicting the region $\eta_1([s_0,\infty)\times[0,\frac{1}{2}))$. The cutoff $\chi_{t_0}$ is supported on the dashed portion of the hyperboloid $\eta_1(\{s_0\}\times[0,\frac{1}{2}))$ in the bottom left corner of the picture. The solution produced is guaranteed to vanish in the dark gray region and is potentially nonzero in the light gray region.}
				\label{eta1_region}
			\end{figure}
		
		Of course, the spacetime region $\eta_1([s_0,\infty)\times[0,\frac{1}{2}))$ does not contain all of $\Lambda_{t_0}$. In order to extend the domain of $u_\text{lin}$, we first remember that along the initial hyperboloid the cutoff $\chi_{t_0}$ is designed to vanish for $y\geq y_0$. Thus, the uniqueness of solutions of linear wave equations, see Lemma 12.8 of \cite{R09} for instance, guarantees that $u_\text{lin}$ vanishes in the dark gray portion of $\eta_1([s_0,\infty)\times[0,\frac{1}{2}))$ depicted in Figure \ref{eta1_region}. With this in mind, we smoothly extend $u_\text{lin}(0,\cdot)$ by zero, i.e., we define functions 				
			\begin{equation}
				Y_1(r):=
				\begin{cases}
					u_\text{lin}(0,r)&r\leq r_0
					\\
					0&r\geq r_0
				\end{cases} \label{U1}
			\end{equation}
		and
			\begin{equation}
				Y_2(r):=
				\begin{cases}
					\partial_0u_\text{lin}(0,r)&r\leq r_0
					\\
					0&r\geq r_0
				\end{cases}. \label{U2}
			\end{equation}
		We can then use these functions as initial data for the Cauchy problem
			$$
			\begin{cases}
				\Big(\partial_t^2-\partial_r^2-\frac{d-1}{r}\partial_r\Big)u(t,r)=2u_1^*(t,r)u(t,r)&(t,r)\in\Lambda_{t_0}
				\\
				u[0](r)=\big(Y_1(r),Y_2(r)\big)&r\in[0,\infty)
			\end{cases}
			$$
		which is guaranteed to have a unique smooth solution since $u_1^*$ is smooth in $\Lambda_{t_0}$ and $Y_1,Y_2$ are smooth, see for instance Theorem 3.2 of \cite{S08}. Since this solution and the original agree on the initial hyperboloid, they must agree wherever $\Lambda_{t_0}$ and $\eta_1([s_0,\infty)\times[0,\frac{1}{2}))$ intersect. Thus, we will have achieved extending $u_\text{lin}$ to a larger region of spacetime which, in fact, strictly contains $\Lambda_{t_0}$, see Figure \ref{extension_region}.
			\begin{figure}
				\includegraphics[scale=0.75]{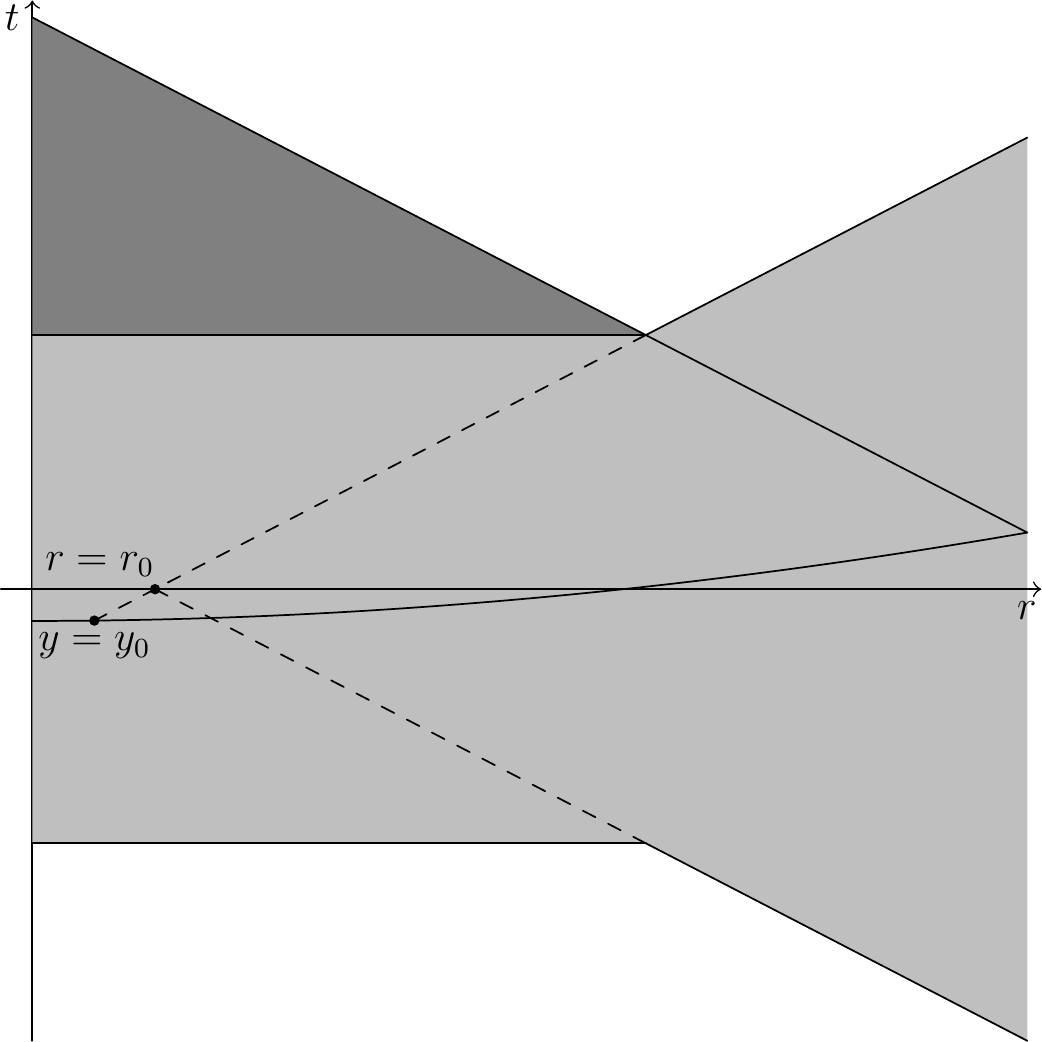}
				\centering
				\caption{A spacetime diagram depicting the domain of the smoothly extended $u_\text{lin}$. The darker region depicts a portion of where $u_\text{lin}$ was originally defined. The lighter region depicts where $u_\text{lin}$ has been extended by solving Equation \eqref{linearized CP}.}
				\label{extension_region} 
			\end{figure}
		Consequently, we have a solution of the linearized equation in $C^\infty(\Lambda_{t_0})$ which satisfies our two conditions. 

Finally, we evolve small, smooth, radial perturbations of $u_1^*$ of the form $(f,g)+\alpha(Y_1,Y_2)$ according to Equation \eqref{qwe}. Adjusting the size of $(f,g)$ and of $|\alpha|$ allows us to ensure $(f,g)+\alpha(Y_1,Y_2)$ is of unit size, i.e., data of the form $u_1^*[0]+(f,g)+\alpha(Y_1,Y_2)$ can be evolved in $\Lambda_{t_0}$ via the quadratic wave equation. Writing the solution as $u=u_1^*+\alpha u_\text{lin}+\varphi_{f,g,\alpha}$, we are able to prove the required bounds on the remainder term $\varphi_{f,g,\alpha}$. Then, by allowing $\alpha$ and $T$ to vary, we are indeed able to run the necessary fixed point argument proving the vanishing of the correction term. Thus, we are able to successfully continue the evolution into all of $\Omega_{T,R}$.
		
		\subsection{Local Existence for Perturbations of $u_1^*$ at $t=0$}
			In this section, we prove a local well-posedness result  in a truncated light cone for sufficiently smooth perturbations of $u_1^*$ of unit size. This will allow us to define the spacetime region $\Lambda_{t_0}$ which will set the stage for proving the main result. We will study strong $H^k$-solutions in light cones of the quadratic wave equation and refer the reader to \cite{BDS19} as well as to  Appendix \ref{Cauchy_Theory}  for the basic definitions. To begin, we define the following function spaces.
			
	\begin{definition}
					Let $k\in\mathbb N_0$, $T>0$, and $T'\in(0,T)$. We define a Banach space $X^k_T(T')$ which consists of functions
						$$
							u:\bigcup_{t\in[0,T']}\{t\}\times\mathbb B^7_{T-t}\to\mathbb R
						$$
					such that $u(t,\cdot)\in H^k(\mathbb B^7_{T-t})$ for each $t\in[0,T']$ and the map $t\to\|u(t,\cdot)\|_{H^k(\mathbb B^7_{T-t})}$ is continuous on $[0,T']$. Furthermore, we set
						$$
							\|u\|_{X^k_T(T')}:=\max_{t\in[0,T']}\|u(t,\cdot)\|_{H^k(\mathbb B^7_{T-t})}.
						$$
				\end{definition}
The proof of the next result is a standard fixed-point argument. However, the specific choice of the life span $t_0$ will be important later on, hence we make it explicit. Furthermore, due to the desire to use standard arguments, we are forced to switch back and forth between radial and non-radial representations of functions on spacetime. To avoid confusion, we will point out explicitly when any identifications are being made.
					
\begin{lemma} \label{local unit solution}
There exists $t_0\in(0,\frac{4}{9})$ such that for all $(f,g)\in H^{10}(\mathbb B_1^7)\times H^9(\mathbb B_1^7)$ satisfying
$$
\|f\|_{H^{10}(\mathbb B_1^7)}+\|g\|_{H^9(\mathbb B_1^7)}\leq1,
						$$
					the initial value problem
						$$
							\begin{cases}
								\big(\partial_t^2-\Delta_x\big)u(t,x)=u(t,x)^2&(t,x)\in\bigcup_{s\in[0,t_0]}\{s\}\times\mathbb B^7_{1-|s|},\label{Cauchy problem}
								\\
								u[0](x)=u_1^*[0](|x|)+\big(f(x),g(x)\big)&x\in\mathbb B_1^7,
							\end{cases}
						$$
					has a unique strong $H^{10}$-solution $u \in X^{10}_{1}(t_0)$ in the truncated light cone $\bigcup_{t\in[0,t_0]}\{t\}\times\mathbb B_{1-t}^7$. Furthermore, $\partial_0 u \in X^{9}_{1}(t_0)$ and the data to solution map is Lipschitz-continuous from $H^{10}(\mathbb B_1^7) \times H^9(\mathbb B_1^7)$ to $X^{10}_{1}(t_0) \times  X^{9}_{1}(t_0)$.
				\end{lemma}
				\begin{proof}
					We solve the Cauchy problem
						$$
							\begin{cases}
								\big(\partial_t^2-\Delta_x\big)\varphi(t,x)=\varphi(t,x)^2+2u_1^*(t,|x|)\varphi(t,x)&(t,x)\in\bigcup_{s\in[0,t_0]}\{s\}\times\mathbb B^7_{1-|s|}
								\\
								\varphi[0](x)=\big(f(x),g(x)\big)&x\in\mathbb B_1^7
							\end{cases}
						$$
in a truncated light cone $\bigcup_{t\in[0,t_0]}\{t\}\times\mathbb B^d_{1-t}$. This yields a solution of the original Cauchy problem by setting $u(t,x)=u_1^*(t,|x|)+\varphi(t,x)$. For $t_0>0$, set 
						$$
							Y(t_0):=\{\varphi\in X_1^{10}(t_0):\|\varphi\|_{X_1^{10}(t_0)}\leq2\gamma_0\}
						$$
					where $\gamma_0:=\max_{s\in[0,\frac{1}{2}]}\gamma_7(1-s)$ and $\gamma_7$ is the continuous function from Proposition \ref{finite speed of propagation}. Define a map $\mathcal K_{f,g}$ on $Y(t_0)$ by
						$$
							\mathcal K_{f,g}(\varphi)(t)=\cos(t|\nabla|)f+\frac{\sin(t|\nabla|)}{|\nabla|}g+\int_0^t\frac{\sin\big((t-s)|\nabla|\big)}{|\nabla|}\mathcal N\big(s,\cdot,\varphi(s,\cdot)\big)ds
						$$
					where
						$$
							\mathcal N\big(t,x,\varphi(t,x)\big)=\varphi^2(t,x)+2u_1^*(t,|x|)\varphi(t,x).
						$$
					By the Banach algebra property, we infer the existence of a constant $0<C_0<\infty$ such that
						\begin{equation}
							\| u v \|_{H^{9}(\mathbb B_r^7)}\leq C_0 \|u \|_{H^{9}(\mathbb B_r^7)}\|v \|_{H^{9}(\mathbb B_r^7)} \label{banach algebra}
						\end{equation}
					for any $u, v \in H^{9}(\mathbb B_r^7)$ and $r\in[\frac{1}{2},1]$. As a consequence, using the bounds stated in Proposition \ref{finite speed of propagation}, we have
						\begin{align*}
							\|\mathcal K_{f,g}(\varphi)(t)\|_{H^{10}(\mathbb B_{1-t}^7)}\leq&\gamma_0\|f\|_{H^{10}(\mathbb B_{1}^7)}+\gamma_0\|g\|_{H^{9}(\mathbb B_{1}^7)}+\gamma_0\int_0^t\|\mathcal N(s,\cdot,\varphi(s,\cdot))\|_{H^{9}(\mathbb B_{1-s}^7)}ds
							\\
							\leq&\gamma_0+\gamma_0\int_0^t\big(\|\varphi(s,\cdot)^2\|_{H^{9}(\mathbb B_{1-s}^7)}+2\|u_1^*(s,|\cdot|)\varphi(s,\cdot)\|_{H^{9}(\mathbb B_{1-s}^7)}\big)ds
							\\
							\leq&\gamma_0+\gamma_0 C_0t_0\big(\|\varphi\|^2_{X_1^{10}(t_0)}+2\|u_1^*\|_{X_1^{10}(t_0)}\|\varphi\|_{X_1^{10}(t_0)}\big)
							\\
							\leq&\gamma_0+\gamma_0 C_0t_0\big((2\gamma_0)^2+4\|u_1^*\|_{X_1^{10}(\frac{4}{9})}\gamma_0\big)
						\end{align*}
Thus, we must have that $\mathcal K_{f,g}:Y(t_0)\to Y(t_0)$ provided $t_0$ is small enough so that the inequality
						$$
							t_0 C_0\big((2\gamma_0)^2+4\|u_1^*\|_{X_1^{10}(\frac{4}{9})}\gamma_0\big)\leq1
						$$
					holds. Similarly, given $\varphi,\psi\in Y(t_0)$, we have
						\begin{align*}
							\|\mathcal K_{f,g}(\varphi)(t)&-\mathcal K_{f,g}(\psi)(t)\|_{H^{10}(\mathbb B_{1-t}^7)}
							\\\
							\leq&\gamma_0\int_0^t\|\mathcal N(s,\cdot,\varphi(s,\cdot))-\mathcal N(s,\cdot,\psi(s,\cdot))\|_{H^{9}(\mathbb B_{1-s}^7)}ds
							\\
							\leq&\gamma_0\int_0^t\big(\|\varphi(s,\cdot)^2-\psi(s,\cdot)^2\|_{H^{9}(\mathbb B_{1-s}^7)}
+2\|u_1^*(s,|\cdot|)(\varphi(s,\cdot)-\psi(s,\cdot))\|_{H^{9}(\mathbb B_{1-s}^7)}\big)ds
							\\
							\leq&\gamma_0 C_0t_0\big(4\gamma_0+2\|u_1^*\|_{X_1^{10}(\frac{4}{9})}\big)\|\varphi-\psi\|_{X_1^{10}(t_0)}.
						\end{align*}
					Thus, $\mathcal K_{f,g}$ is Lipschitz on $Y(t_0)$ with Lipschitz constant at most $\frac{1}{2}$ provided $t_0$ is small enough so that the inequality
						$$
							\gamma_0 t_0 C_0\big(4\gamma_0+2\|u_1^*\|_{X_1^{10}(\frac{4}{9})}\big)\leq\frac{1}{2}
						$$
holds. Thus, $\mathcal K_{f,g}$ is a contraction on the closed subspace $Y(t_0)$ of the Banach space $X_1^{10}(t_0)$. As a consequence, the Banach fixed point theorem implies the existence of a unique fixed point $\varphi_{f,g}\in Y(t_0)$ of $\mathcal K_{f,g}$. This and standard arguments on unconditional uniqueness prove the existence of $u$ with the claimed properties. The statement about the time derivative follows from the Duhamel formula and the bounds stated in Proposition \ref{finite speed of propagation}. The Lipschitz dependence on the data is again standard. 
\end{proof}
				
		\subsection{Construction of an Adjustment Term}\label{Construction of an Adjustment Term}
In this section, we construct the functions $Y_1$ and $Y_2$ as described in Section \ref{Motivation and Overview for the Construction of (Y_1,Y_2)}. We begin by showing that $\mathbf P_4$ applied to suitably truncated versions of $\mathbf f_4^*$ does not vanish. Then, we take one of these suitably truncated versions of $\mathbf f_4^*$ and, starting on a specific hyperboloid, evolve it according to the quadratic wave equation linearized around $u_1^*$ into the region $\Lambda_{t_0}$. 

For the rest of the paper,  $t_0$  denotes the constant from Lemma \ref{local unit solution}. With $r_0:=\frac{t_0}{4}$, we define the number	
				\begin{equation}
					y_0:=e^{s_0}\frac{(4+\sqrt{2})r_0^2+4(2+\sqrt{2})r_0}{8(r_0+2+\sqrt{2})}. \label{yT0}
				\end{equation}
			This number is chosen by solving the equation $r-r_0=1+e^{-s_0}h(e^{s_0}r)$ for $r$ and setting $y_0=e^{s_0}r$. For reasons to be made clear soon, consider the function $F: [0,\frac{1}{2}] \to \R$ defined by
				\begin{align}\label{Function_F}
		F(y)=\frac{f_{4,1}(y)}{W(y;4)}\frac{\tilde G_4(y)}{c_{12}(y)}-\left[\frac{c_{21}(f_{4,1}^*)^2}{W(\cdot;4)c_{12}}\right]'(y)
				\end{align}

where 
				$$
					\tilde G_4(y)=c_{21}(y)[f_{4,1}^*]'(y)+\big(c_{20}(y)-12\big)f_{4,1}^*(y)
				$$
			and
				$$
					W(y;4)=\frac{\left(y^2+1\right) \left(1-4 y^2\right)^{-4 } \left(3
   \sqrt{y^2+2}-y+4\right)^{2} \left(3 \sqrt{y^2+2}+y+4\right)^{2}}{y^6
   \sqrt{y^2+2} \sqrt{y^2+2 \sqrt{y^2+2}+3}}.
				$$
			A straightforward calculation shows that $F(0) = 0$ and that there exists $\delta_0>0$ so that $F(y)>0$ for $y\in(0,\delta_0)$. Now, let $\chi_{t_0}:[0,\infty)\to[0,1]$ be a smooth cutoff function with $\chi_{t_0}(y)=1$ for $y\leq\frac{1}{2}\min\{y_0,\delta_0\}$, $\chi_{t_0}(y)=0$ for $y\geq\frac{3}{4}\min\{y_0,\delta_0\}$. With this, we state and prove the following crucial lemma.
			\begin{lemma} \label{nonvanishing of the Riesz projection}
				We have $(\mathbf P_4(\chi_{t_0}\mathbf f_4^*)|\mathbf f_4^*)_{\mathcal H_R}\neq0$.
			\end{lemma}
			\begin{proof}
				Observe that $\chi_{t_0}\mathbf f_4^*\in C_e^\infty[0,R]^2$. By Proposition \ref{projection}, we must have $\mathbf P_4(\chi_{t_0}\mathbf f_4^*)=c\mathbf f_4^*$ for some $c\in\mathbb C$. Thus, $(\mathbf P_4(\chi_{t_0}\mathbf f_4^*)|\mathbf f_4^*)_{\mathcal H_R}=c\|\mathbf f_4^*\|_{\mathcal H_R}^2$. What is nontrivial, however, is to show that $c\neq0$. To show $c\neq0$, we turn our attention to investigating the first component of this Riesz projection.
				
				For $\lambda\in\rho(\mathbf L)$ the first component of $\mathbf u:=\mathbf R_\mathbf L(\lambda)(\chi_{t_0}\mathbf f_4^*)$ solves the ODE
					\begin{align}
					\begin{split}
						 u''(y)&+\frac{c_{11}(y)+(\lambda+2)c_{21}(y)}{c_{12}(y)}u'(y)+\frac{(\lambda+2)(c_{20}(y)-\lambda-2)+ V(y)}{c_{12}(y)}u(y)
						 \\
						 =&\frac{\chi_{t_0}(y)\tilde G_\lambda(y)+c_{21}(y)\chi_{t_0}'(y)f_{4,1}^*(y)}{c_{12}(y)} \label{inhom spectral eqn}
					\end{split}
					\end{align}
				on the interval $(0,R)$ where
					$$
						\tilde G_\lambda(y)=c_{21}(y)[f_{4,1}^*]'(y)+\big(c_{20}(y)-8-\lambda\big)f_{4,1}^*(y).
					$$
				For the homogeneous equation, the Frobenius indices at the singular point $y=0$ are $\{0,-5\}$ while at $y=\frac{1}{2}$ they are $\{0,1-\lambda\}$. Denote by $\phi_1(\cdot;\lambda)$ a solution of the homogeneous version of Equation \eqref{inhom spectral eqn} taking the index $-5$ at $y=0$ and $\phi_0(\cdot;\lambda)$ a solution of the homogeneous version of Equation \eqref{inhom spectral eqn} taking the index $0$ at $y=0$. Observe that $\phi_0(\cdot;\lambda)$ must take the index $1-\lambda$ at $y=\frac{1}{2}$ since, if otherwise, $\phi_0(\cdot;\lambda)\in C^\infty[0,1]$ and this is excluded by Proposition \ref{spectrum} for $\lambda\in\rho(\mathbf L)$. The Wronskian can be expressed as $W\big(\phi_1(\cdot;\lambda),\phi_0(\cdot;\lambda)\big)(y)=C(\lambda)W(y;\lambda)$ where $C(\lambda)$ is some constant depending on $\lambda$ and 
					$$
						W(y;\lambda)=\frac{\left(y^2+1\right) \left(1-4 y^2\right)^{-\lambda } \left(3
   \sqrt{y^2+2}-y+4\right)^{\lambda /2} \left(3 \sqrt{y^2+2}+y+4\right)^{\lambda /2}}{y^6
   \sqrt{y^2+2} \sqrt{y^2+2 \sqrt{y^2+2}+3}}.
					$$
				The fact that $\lambda=4$ is a simple pole of the resolvent implies that $C(\lambda)$ must vanish to order one at $\lambda=4$.
				
				Since neither of these two fundamental solutions live in $H_\text{rad}^6(\mathbb B^7_{1/2})$, variation of parameters implies
					\begin{align*}
						[\mathbf R_{\mathbf L}(\lambda)&(\chi_{t_0}\mathbf f_4^*)]_1(y)
						\\
						=&\phi_1(y;\lambda)\int_0^y\frac{\phi_0(\tilde y;\lambda)}{W\big(\phi_0(\cdot;\lambda),\phi_1(\cdot;\lambda)\big)(\tilde y)}\frac{\chi_{t_0}(\tilde y)\tilde G_\lambda(\tilde y)+c_{21}(\tilde y)\chi_{t_0}'(\tilde y)f_{4,1}^*(\tilde y)}{c_{12}(\tilde y)}d\tilde y
						\\
						&+\phi_0(y;\lambda)\int_y^{\frac{1}{2}}\frac{\phi_1(\tilde y;\lambda)}{W\big(\phi_0(\cdot;\lambda),\phi_1(\cdot;\lambda)\big)}\frac{\chi_{t_0}(\tilde y)\tilde G_\lambda(\tilde y)+c_{21}(\tilde y)\chi_{t_0}'(\tilde y)f_{4,1}^*(\tilde y)}{c_{12}(\tilde y)}d\tilde y.
					\end{align*}
				By repeated integration by parts, one can indeed see that the above expression is in $H_\text{rad}^6(\mathbb B^7_{1/2})$.

Since $\mathbf f_4^*$ is an eigenfunction, we must have that $\phi_0(\cdot;4)$ and $\phi_1(\cdot;4)$ are multiples of $f^*_{4,1}$. Consequently, we obtain
					$$
						[\mathbf P_4(\chi_{t_0}\mathbf f_4^*)]_1(y)=Cf_{4,1}^*(y)\int_0^\frac{1}{2}\frac{f_{4,1}^*(\tilde y)}{W(\tilde y;4)}\frac{\chi_{t_0}(\tilde y)\tilde G_4(\tilde y)+c_{21}(\tilde y)\chi_{t_0}'(\tilde y)f_{4,1}^*(\tilde y)}{c_{12}(\tilde y)}d\tilde y
					$$
				for some $C\in\mathbb C\setminus\{0\}$. An integration by parts yields
					\begin{align*}
						\int_0^\frac{1}{2}&\frac{f_{4,1}^*(\tilde y)}{W(\tilde y;4)}\frac{\chi_{t_0}(\tilde y)\tilde G_4(\tilde y)+c_{21}(\tilde y)\chi_{t_0}'(\tilde y)f_{4,1}^*(\tilde y)}{c_{12}(\tilde y)}d\tilde y
=\int_0^\frac{1}{2}\chi_{t_0} F(y) d\tilde y,
					\end{align*}
with $F$ defined in \eqref{Function_F}. By definition of $\chi_{t_0}$, the integrand is positive within $\supp(\chi_{t_0})$ which implies
					$$
						[\mathbf P_4(\chi_{t_0}\mathbf f_4^*)]_1(y)=c f_{4,1}^*(y)
					$$ 
				for some $c\neq0$. This implies the claim.
			\end{proof}

			\begin{lemma} \label{adjustment term}
There exists a smooth, radial solution $u$ of Equation \eqref{Eq:QWE_Linearized} with $T=1$ in the spacetime region $\Lambda_{t_0}$ with the following two properties:
					\begin{enumerate}
						\item $(u\circ\eta_1)(s_0,y)=e^{-2s_0}\chi_{t_0}(y)f_{4,1}^*(y)$ and $\partial_0(u\circ\eta_1)(s_0,y)=e^{-2s_0}\chi_{t_0}(y)f_{4,2}^*(y)$ for all $y\in[0,\infty)$ and
						\item $u|_{t=0}$ and $\partial_0u|_{t=0}$ have support contained in the interval $[0,r_0)$.
					\end{enumerate}
			\end{lemma}
			\begin{proof}
Let $k\in\mathbb N$, $k\geq6$ and consider the abstract initial value problem
					$$
					\begin{cases}
						\partial_s\Phi(s)=\mathbf  L\Phi(s)
						\\
						\Phi(s_0)=\chi_{t_0}\mathbf  f_4^*
					\end{cases}
					$$
				on the space $\mathcal H_{1/2}^k$. For $k=6$, we have that the unique solution in $C([s_0,\infty),\mathcal H^6_{1/2})$ is given by 
					$$
						\Phi(s):=\mathbf S(s-s_0)(\chi_{t_0}\mathbf  f_4^*) \quad \text{for } s\geq s_0
					$$
				where $\big(\mathbf S(s)\big)_{s\geq0}$ is the semigroup from Lemma \ref{well-posed}. For $k>6$, Lemma \ref{generation} along with the bounded perturbation theorem imply that $\mathbf L$ is also the generator of a semigroup on $\mathcal H_{1/2}^k$ which we will denote by $\big(\mathbf S_k(s)\big)_{s\geq0}$. Observe that for any $k\geq6$, $\mathbf S(s)|_{\mathcal H_{1/2}^k}=\mathbf S_k(s)$ for all $s\geq0$ following the argument of Lemma 3.5 of \cite{CGS21}. Since $\chi_{t_0}\mathbf  f_4^*\in C_e^\infty[0,R]^2$, we have as a consequence that $\Phi$ is the unique solution in $C([s_0,\infty),\mathcal H^k_{1/2})$ for any $k\geq6$. Consequently, $\Phi(s)\in\mathcal H^k_{1/2}$ for any $k\in\mathbb N$, $k\geq6$ and $s\geq s_0$ which, by Sobolev embedding, implies $\Phi(s)(|\cdot|)\in C^\infty(\mathbb B_{1/2}^7)^2$ for all $s\geq s_0$.
				
Since $\chi_{t_0}\mathbf  f_4^*\in C_e^\infty[0,R]^2$, Theorem 6.1.5 of \cite{P83} implies that $\Phi\in C^1([s_0,\infty),\mathcal H^k_{1/2})$ and
					\begin{equation}
						\partial_s\Phi(s)=\mathbf  S(s-s_0)\big(\mathbf  L(\chi_{t_0}\mathbf  f_4^*)\big). \label{derivative of eqn}
					\end{equation}
				Thus, $\partial_s\Phi(s)(|\cdot|)\in C^\infty(\mathbb B_{1/2}^7)^2$ for all $s\geq s_0$. From Equation \eqref{derivative of eqn}, we find that $\partial_s\Phi\in C^1([s_0,\infty),\mathcal H^k_{1/2})$. Thus, by an inductive argument, we conclude that for all $m\in\mathbb N$, $\Phi\in C^m([s_0,\infty),\mathcal H^k_{1/2})$ with $\partial_s^m\Phi(s)=\mathbf S(s-s_0)\big(\mathbf  L^m(\chi_{t_0}\mathbf  f_4^*)\big)$. Again by Sobolev embedding, $\partial_s^m\Phi(s)(|\cdot|)\in C^\infty(\mathbb B_{1/2}^7)^2$ for all $m\in\mathbb N$ and $s\geq s_0$. Since all $s$ and $y$ derivatives exist to arbitrary order and are continuous, we conclude that $\Phi\in C^\infty([s_0,\infty)\times[0,\frac{1}{2}))^2$.
				
				Upon defining $v(s,y):=e^{2s}\phi_1(s,y)$, we have that 
	\[u_\text{lin}(t,r):=\big(v\circ\eta_1^{-1}\big)(t,r)\]
is a smooth solution of Equation \eqref{Eq:QWE_Linearized} with $T=1$ in the spacetime region $\eta_1([s_0,\infty)\times[0,\frac{1}{2}))$ satisfying the first of the two claimed properties. To begin extending the domain of $u_\text{lin}$ we define functions 
					$$
						Y_1(r):=
						\begin{cases}
							u_\text{lin}(0,r)&r\leq r_0,
							\\
							0&r\geq r_0,
						\end{cases} 
					$$
				and
					$$
						Y_2(r):=
						\begin{cases}
							\partial_0u_\text{lin}(0,r)&r\leq r_0,
							\\
							0&r\geq r_0.
						\end{cases}.
					$$
			Clearly, for $r>r_0$ all derivatives of $Y_1$ and $Y_2$ vanish. For $r<r_0$, all derivatives of $Y_1$ and $Y_2$ exist and are continuous since, for such $r$, both functions are the composition of smooth functions. All $r$-derivatives of $u_\text{lin}(0,\cdot)$ from the left vanish and $u_\text{lin}(0,r_0)$ by the choice of cutoff $\chi_{t_0}$ implying smoothness of $Y_1$ and $Y_2$ at $r=r_0$.
				
				Now, consider the Cauchy problem 
					\begin{equation}
						\begin{cases}
							\Big(\partial_t^2-\partial_r^2-\frac{d-1}{r}\partial_r\Big)u(t,r)=2u_1^*(t,r)u(t,r)&(t,r)\in\Lambda_{t_0}
							\\
							u[0](r)=\big(Y_1(r),Y_2(r)\big)&r\in[0,\infty) \label{linearized CP}
						\end{cases}.
					\end{equation}
				Since $u_1^*\in C^\infty(\Lambda_{t_0})$ and is radial along with $Y_1(|\cdot|),Y_2(|\cdot|)\in C^\infty(\mathbb R^7)$, the Cauchy problem \eqref{linearized CP} has a unique radial solution $u\in C^\infty(\Lambda_{t_0})$. Furthermore, since the data satisfies $\supp(Y_1,Y_2)\subseteq[0,r_0)$ and is specified at $t=0$, finite speed of propagation implies that $u=0$ in the spacetime region $\{(t,r)\in\mathbb R\times[0,\infty):r_0-r\leq t\leq r_0+r,\;r\geq r_0\}$. Finally, uniqueness of solutions to linear wave equations allows us to conclude that the solution we have just produced satisfies $u(t,r)=u_\text{lin}(t,r)$ for $(t,r)\in\eta_1([s_0,\infty)\times[0,\frac{1}{2}))$, i.e., the solution $u$ smoothly extends $u_\text{lin}$. Consequently, we have a function $u_\text{lin}\in C^\infty(\Lambda_{t_0})$ which solves Equation \eqref{Eq:QWE_Linearized} and satisfies 
		\[ (u_\text{lin}\circ\eta_1)(s_0,y)=e^{-2s_0}\chi_{t_0}(y)f_{4,1}^*(y)\text{ and }\partial_0(u_\text{lin}\circ\eta_1)(s_0,y)=e^{-2s_0}\chi_{t_0}(y)f_{4,2}^*(y).\] 
			\end{proof}
		
		\section{Evolution of Physical Initial Data} \label{Evolution of Physical Initial Data}
		
In this section, we evolve sufficiently small, smooth, radial perturbations of $u_1^*$ adjusted by some multiple of $u_\text{lin}$ in the spacetime region $\Lambda_{t_0}$ according to the quadratic wave equation. In this region, we will obtain uniform control of sufficiently many derivatives of the evolution in order to continue it via the hyperboloidal formulation. For convenience, we define
\[
					\mathcal B_{\delta,M_0}:=\{(f,g)\in C^{\infty}(\B^7)^2\text{ radial}:\supp(f,g)\subset \B^7_{r_0},\|(f,g)\|_{H^{10}(\mathbb R^7)\times H^{9}(\mathbb R^7)}\leq \tfrac{\delta}{M_0^2}\}
\]
			for $\delta,M_0>0$. Furthermore, due to the need to switch between the radial and non-radial picture, we also define the spacetime region
				\begin{align}\label{Def:Rotated_Lambda_Region}
					\overline\Lambda_{t_0}:=[-t_0,t_0]\times\R^7\cup\{(t,x)\in\mathbb R\times\R^7:-|x|+r_0\leq t\leq |x|-r_0\}.
				\end{align}
			
			\begin{lemma} \label{local adjusted CP}
				Let $t_0>0$ be as in Lemma \ref{local unit solution}. For all sufficiently small $\delta>0$ and sufficiently large $M_0>0$, we have that for all $(f,g)\in\mathcal B_{\delta,M_0}$ and $|\alpha|\leq\frac{\delta}{M_0}$, the initial value problem
					\begin{equation}
					\begin{cases}
						\big(\partial_t^2-\Delta_x\big)u(t,x)=u(t,x)^2&(t,x)\in\overline\Lambda_{t_0}
						\\
						u[0](x)=u_1^*[0](|x|)+\alpha u_\text{\normalfont{lin}}[0](|x|)+\big(f(|x|),g(|x|)\big) \label{adjusted CP}&x\in\R^7
					\end{cases}
					\end{equation}
				has a unique radial solution $u_{f,g,\alpha}\in C^\infty(\overline\Lambda_{t_0})$ of the form $$u_{f,g,\alpha}(t,x)=u_1^*(t,|x|)+\alpha u_\text{\normalfont{lin}}(t,|x|)+\varphi_{f,g,\alpha}(t,x)$$ which depends continuously on $\alpha$ and
					\begin{align}
						\sup_{(t,x)\in\overline\Lambda_{t_0}}|\partial^\kappa\varphi_{f,g,\alpha}(t,x)|\lesssim\frac{\delta}{M_0^2} \label{uniform control}
					\end{align}
				for all multi-indices $\kappa\in\mathbb N_0^8$ with $|\kappa|\leq6$.
			\end{lemma}
			\begin{proof}
				Since $u_\text{lin}[0]$ and $(f,g)$ are supported in the interval $[0,r_0)$, we can ensure that for all $\delta>0$ sufficiently small and $M_0>0$ sufficiently large, $\|\alpha u_\text{lin}[0]+(f,g)\|_{H^{10}(\mathbb R^7)\times H^9(\mathbb R^7)}\leq1$ holds. Thus, by applying Lemma \ref{local unit solution}, Equation \eqref{adjusted CP} has a unique strong $H^{10}$-solution in the truncated light cone up to time $t_0$. It remains to show that our solution is of the stated form. To that end, we solve the equation
					\begin{align}
					\begin{split}
						\big(\partial_t^2-\Delta_x\big)\varphi(t,x)=\varphi(t,x)^2+2\big(u_1^*(t,|x|)+\alpha u_\text{lin}(t,|x|)\big)\varphi(t,x)+\alpha^2u_\text{lin}(t,|x|)^2 \label{lin CP}
					\end{split}
					\end{align}
				for $(t,x)\in\bigcup_{s\in[0,t_0]}\{s\}\times\mathbb B^7_{1-t}$ with initial data
					$
						\varphi[0]=(f,g).
					$
				 A strong $H^{10}$-solution of \eqref{lin CP} yields a strong $H^{10}$-solution of the original Cauchy problem by setting $u(t,x)=u_1^*(t,|x|)+\alpha u_\text{lin}(t,|x|)+\varphi(t,x)$. Set 
					$$
						Y'(t_0):=\bigg\{\varphi\in X_1^{10}(t_0):\|\varphi\|_{X_1^{10}(t_0)}\leq2\gamma_0\frac{\delta}{M_0^2}\bigg\}.
					$$
				Define a map $\mathcal K_{f,g,\alpha}$ on $Y'(t_0)$ by
					$$
						\mathcal K_{f,g,\alpha}(\varphi)(t):=\cos(t|\nabla|)f+\frac{\sin(t|\nabla|)}{|\nabla|}g+\int_0^t\frac{\sin\big((t-s)|\nabla|\big)}{|\nabla|}\mathcal N_\alpha\big(s,\cdot,\varphi(s,\cdot)\big)ds,\;t\in[0,t_0]
					$$
				where
					$$
						\mathcal N_\alpha\big(t,x,\varphi(t,x)\big):=\varphi^2(t,x)+2(u_1^*(t,|x|)+\alpha u_\text{lin}(t,|x|))\varphi(t,x)+\alpha^2u_\text{lin}^2(t,|x|).
					$$
				Similar to the proof of Lemma \ref{local unit solution}, we have
					\begin{align*}
						\|\mathcal K_{f,g,\alpha}&(\varphi)(t)\|_{H^{10}(\mathbb B_{1-t}^7)}
						\\
						\leq&\gamma_0\|f\|_{H^{10}(\mathbb B_{1}^7)}+\gamma_0\|g\|_{H^{9}(\mathbb B_{1}^7)}+\gamma_0\int_0^t\|\mathcal N_\alpha(s,\cdot,\varphi(s,\cdot))\|_{H^{9}(\mathbb B_{1-s}^7)}ds
						\\
						\leq&\gamma_0\frac{\delta}{M_0^2}+\gamma_0 C_0\int_0^t\bigg(\|\varphi(s,\cdot)\|^2_{H^{9}(\mathbb B_{1-s}^7)}
						\\
						&+2(\|u_1^*(s,|\cdot|)\|_{H^{9}(\mathbb B_{1-s}^7)}+|\alpha|\|u_\text{lin}(s,|\cdot|)\|_{H^{9}(\mathbb B_{1-s}^7)})\|\varphi(s,\cdot)\|_{H^{9}(\mathbb B_{1-s}^7)}
						\\
						&+\alpha^2\|u_\text{lin}(s,|\cdot|)\|^2_{H^{9}(\mathbb B_{1-s}^7)}\bigg)ds
						\\
						\leq&\gamma_0\frac{\delta}{M_0^2}+\gamma_0 C_0t_0\bigg(\|\varphi\|^2_{X_1^{10}(t_0)}+2(\|u_1^*\|_{X_1^{10}(\frac{4}{9})}+|\alpha|\|u_\text{lin}\|_{X_1^{10}(t_0)})\|\varphi\|_{X_1^{10}(t_0)}
						\\
						&+\alpha^2\|u_\text{lin}\|^2_{X_1^{10}(t_0)}\bigg)
						\\
						\leq&\gamma_0\frac{\delta}{M_0^2}+\gamma_0 C_0t_0\bigg((2\gamma_0\frac{\delta}{M_0^2})^2+2\left(\|u_1^*\|_{X_1^{10}(\frac{4}{9})}+\frac{\delta}{M_0}\|u_\text{lin}\|_{X_1^{10}(t_0)}\right )2\gamma_0\frac{\delta}{M_0^2}
						\\
						&+\big(\frac{\delta}{M_0}\big)^2\|u_\text{lin}\|_{X_1^{10}(t_0)}^2\bigg)
					\end{align*}
				where the constant $C_0$ is the same constant as in Equation \eqref{banach algebra}. Thus, it follows that $\mathcal K_{f,g,\alpha}:Y'(t_0)\to Y'(t_0)$ provided the inequality
					$$
						C_0t_0\Big((2\gamma_0)^2\frac{\delta}{M_0^2}+2\big(\|u_1^*\|_{X_1^{10}(\frac{4}{9})}+\frac{\delta}{M_0}\|u_\text{lin}\|_{X_1^{10}(t_0)}\big)2\gamma_0+\delta\big(\|u_\text{lin}\|_{X_1^{10}(t_0)}\big)^2\Big)\leq1
					$$
				holds. Recalling that $t_0>0$ was chosen so that the inequality
					$$
						C_0t_0\big((2\gamma_0)^2+4\|u_1^*\|_{X_1^{10}(\frac{4}{9})} \gamma_0\big)\leq1
					$$
				held, we observe that by considering smaller $\delta$ if necessary, we can ensure that the desired inequality is satisfied. Similarly, given $\varphi,\psi\in Y'(t_0)$, 
					\begin{align*}
						\|\mathcal K_{f,g,\alpha}&(\varphi)(t)-\mathcal K_{f,g,\alpha}(\psi)(t)\|_{H^k(\mathbb B_{1-t}^7)}
						\\
						\leq&\gamma_0\int_0^t\|\mathcal N_\alpha(s,\cdot,\varphi(s,\cdot))-\mathcal N_\alpha(s,\cdot,\psi(s,\cdot))\|_{H^{9}(\mathbb B_{1-s}^7)}ds
						\\
						\leq&\gamma_0\int_0^t\bigg(\|(\varphi(s,\cdot)-\psi(s,\cdot))^2\|_{H^{9}(\mathbb B_{1-s}^7)}
						\\
						&+2\|(u_1^*(s,|\cdot|)+\alpha u_\text{lin}(s,|\cdot|))(\varphi(s,\cdot)-\psi(s,\cdot))\|_{H^{9}(\mathbb B_{1-s}^7)}\bigg)ds
						\\
						\leq&\gamma_0 C_0t_0\Big(4\gamma_0\frac{\delta}{M_0^2}+2\big(\|u_1^*\|_{X_1^{10}(\frac{4}{9})}+\frac{\delta}{M_0}\|u_\text{lin}\|_{X_1^{10}(t_0)}\big)\Big)\|\varphi-\psi\|_{X_1^{10}(t_0)}.
					\end{align*}
				Thus $\mathcal K_{f,g,\alpha}$ is Lipschitz on $Y'(t_0)$ with Lipschitz constant at most $\frac{1}{2}$ provided that the inequality
					$$
						C_0\gamma_0 t_0\Big(4\gamma_0\frac{\delta}{M_0^2}+2\big(\|u_1^*\|_{X_1^{10}(\frac{4}{9})}+\frac{\delta}{M_0}\|u_\text{lin}\|_{X_1^{10}(t_0)}\big)\Big)\leq\frac{1}{2}
					$$
				holds. Again, recalling that $t_0>0$ was chosen so that, in addition, the inequality
					$$
						C_0\gamma_0 t_0\big(4\gamma_0+2\|u_1^*\|_{X_1^{10}(t_0)}\big)\leq\frac{1}{2}
					$$
				held, we observe that by considering smaller $\delta$ if necessary, we can ensure that the desired inequality is satisfied. Thus, $\mathcal K_{f,g,\alpha}$ is a contraction on the closed subspace $Y'(t_0)$ of the Banach space $X_1^{10}(t_0)$. The Banach fixed point theorem implies the existence of a unique fixed point, namely $\varphi_{f,g,\alpha}\in Y'(t_0)$, of $\mathcal K_{f,g,\alpha}$. 
				
				By the same argument, we obtain a solution for $(t,x)\in\bigcup_{s\in[-t_0,0]}\{s\}\times\mathbb B^d_{1-|t|}$ thereby extending the domain of $\varphi_{f,g,\alpha}$ to $(t,x)\in\bigcup_{s\in[-t_0,t_0]}\{s\}\times\mathbb B^d_{1-|t|}$. Moreover, by upgrade of regularity, see Theorem 2.14 of \cite{BDS19}, we infer that $\varphi_{f,g,\alpha}$ is smooth. Finite speed of propagation together with the fact that $(f,g)$ are compactly supported imply that $\varphi_{f,g,\alpha} \in C^\infty(\overline\Lambda_{t_0})$. From Sobolev embedding, and using the equation to estimate time derivatives of higher order, we infer that 
\[  \sup_{(t,x) \in \overline\Lambda_{t_0}} | \partial^{\kappa} \varphi_{f,g,\alpha} (t,x)| \lesssim \frac{\delta}{M_0^2}  \]
for spacetimes derivatives with  $\kappa \in \N_0^8$, $|\kappa| \leq 6$. Since the initial data and coefficients in the equation are radial, the same is true for $\varphi_{f,g,\alpha}$. Finally, setting 
 \[u_{f,g,\alpha}(t,x):=u_1^*(t,|x|)+\alpha u_\text{lin}(t,|x|)+\varphi_{f,g,\alpha}(t,x)\]
and using the information on the support of $\varphi_{f,g,\alpha}$ and $u_\text{lin}$ yields $u_{f,g,\alpha}\in C^\infty(\overline{\Lambda_{t_0}})$. Continuous dependence on $\alpha$ follows from a standard argument.
			\end{proof}
			
		From this point forward, we will abuse notation and not distinguish the functions $u_{f,g,\alpha}$ and $\varphi_{f,g,\alpha}$ from their radial representatives. More precisely, we will think of $u_{f,g,\alpha}$ and $\varphi_{f,g,\alpha}$ as functions on $\Lambda_{t_0}$ instead of functions on $\overline\Lambda_{t_0}$.
		
		\subsection{The Initial Data Operator}
			In this section, we use the solution obtained in Lemma \ref{local adjusted CP} to obtain data on a family of hyperboloids. Then, by a fixed point argument we will show that we can continue its evolution into $\Omega_{T,R}$ using the nonlinear theory developed in Section \ref{Well-Posedness and Decay of the Nonlinear Evolution} for at least one choice of $T$ and $\alpha$. First, we define a map which sends initial data to the restriction of the solution obtained in Lemma \ref{local adjusted CP} on a family of hyperboloids.
			\begin{definition}
				Let $\delta>0$ be sufficiently small and $M_0>0$ be sufficiently large so that, given $(f,g)\in\mathcal B_{\delta,M_0}$ and $\alpha\in[-\frac{\delta}{M_0},\frac{\delta}{M_0}]$, the unique solution of \eqref{adjusted CP}, $u_{f,g,\alpha}\in C^\infty(\Lambda_{t_0})$, exists. Then we set
					$$
						\mathbf U((f,g),\alpha,\beta):=e^{-2s}
						\begin{pmatrix}
							u_{f,g,\alpha}\circ\eta_{1+\beta}-u_{1+\beta}^*\circ\eta_{1+\beta}
							\\
							\partial_s(u_{f,g,\alpha}\circ\eta_{1+\beta})-\partial_s(u_{1+\beta}^*\circ\eta_{1+\beta})
						\end{pmatrix}\Bigg|_{s=s_0}.
					$$
				We call $\mathbf U$ the \textit{initial data operator}.
			\end{definition}
			We have the following mapping properties.
			\begin{lemma} \label{ID operator}
				For all $\delta>0$ sufficiently small and $M_0>0$ sufficiently large, the initial data operator $\mathbf U:\mathcal B_{\delta,M_0}\times[-\frac{\delta}{M_0},\frac{\delta}{M_0}]^2\to\mathcal H_R$ is well-defined and for any $(f,g)\in\mathcal B_{\delta,M_0}$, the map $\mathbf U((f,g),\cdot):[-\frac{\delta}{M_0},\frac{\delta}{M_0}]^2\to\mathcal H_R$ is continuous. Furthermore, there exists $\gamma_{t_0}\in\mathbb R\setminus\{0\}$ such that
					\begin{equation}
						\mathbf U((f,g),\alpha,\beta)=\gamma_{t_0}\beta\mathbf f_1^*+\alpha\chi_{t_0}\mathbf f_4^*+\mathbf V((f,g),\alpha,\beta), \label{initial data operator expansion}
					\end{equation}
				and $\mathbf V((f,g),\alpha,\beta)$ satisfies the bound
					$$
						\|\mathbf V((f,g),\alpha,\beta)\|_{\mathcal H_R}\lesssim\frac{\delta}{M_0^2}+|\alpha|^2+|\beta|^2.
					$$
			\end{lemma}
			\begin{proof}
				The initial data operator is well-defined since the hyperboloids $\eta_{1+\beta}(\{s_0\}\times[0,R))$, $\beta\in[-\frac{\delta}{M_0},\frac{\delta}{M_0}]$, lie entirely in $\Lambda_{t_0}$ for sufficiently small $\delta$ and sufficiently large $M_0$. Continuity of $\mathbf U((f,g),\cdot)$ follows from $u_{1+\beta}^*\in C^\infty(\Lambda_{t_0})$ and continuous dependence on $\alpha$ of $u_{f,g,\alpha}$. To see the stated expansion of $\mathbf U((f,g),\alpha,\beta)$, we insert the form of $u_{f,g,\alpha}$ and group terms as follows
					\begin{align*}
						\mathbf U((f,g),\alpha,\beta)=&
						e^{-2s}\begin{pmatrix}
							(u_1^*+\alpha u_\text{lin}+\varphi_{f,g,\alpha})\circ\eta_{1+\beta}-u_{1+\beta}^*\circ\eta_{1+\beta}
							\\
							\partial_s((u_1^*+\alpha u_\text{lin}+\varphi_{f,g,\alpha})\circ\eta_{1+\beta})-\partial_s(u_{1+\beta}^*\circ\eta_{1+\beta})
						\end{pmatrix}\Bigg|_{s=s_0}
						\\
						=&e^{-2s}\begin{pmatrix}
							u_1^*\circ\eta_{1+\beta}-u_{1+\beta}^*\circ\eta_{1+\beta}
							\\
							\partial_s(u_1^*\circ\eta_{1+\beta})-\partial_s(u_{1+\beta}^*\circ\eta_{1+\beta})
						\end{pmatrix}\Bigg|_{s=s_0}
						\\
						&+\alpha e^{-2s}\begin{pmatrix}
							u_\text{lin}\circ\eta_{1+\beta}-u_\text{lin}\circ\eta_{1}
							\\
							\partial_s(u_\text{lin}\circ\eta_{1+\beta})-\partial_s(u_\text{lin}\circ\eta_{1})
						\end{pmatrix}\Bigg|_{s=s_0}
						\\
						&+\alpha e^{-2s}\begin{pmatrix}
							u_\text{lin}\circ\eta_{1}
							\\
							\partial_s(u_\text{lin}\circ\eta_{1})
						\end{pmatrix}\Bigg|_{s=s_0}
						\\
						&+e^{-2s}\begin{pmatrix}
							\varphi_{f,g,\alpha}\circ\eta_{1+\beta}
							\\
							\partial_s(\varphi_{f,g,\alpha}\circ\eta_{1+\beta})
						\end{pmatrix}\Bigg|_{s=s_0}.
					\end{align*}
				Now, recall 
					$$
						(u_T^*\circ\eta_T)(s,y)=-\frac{24e^{2s}\big(5y^2-21h(y)^2\big)}{\big(3h(y)^2+5y^2\big)^2}
					$$
				which is clearly independent of $T$. Thus, for the first term we can write
					\begin{align*}
						u_1^*(1+\beta&+e^{-s}h(y),e^{-s}y)-u_{1+\beta}^*(1+\beta+e^{-s}h(y),e^{-s}y)
						\\
						&=\partial_Tu_1^*(T+e^{-s}h(y),e^{-s}y)|_{T=1}\beta+r_1(\beta,s,y)\beta^2
					\end{align*}
				where
					$$
						r_1(\beta,s,y)=\int_0^1\Big(\int_0^1\partial_\beta^2u_1^*(1+\beta xz+e^{-s}h(y),e^{-s}y)dz\Big)xdx.
					$$
				We have that $r_1$ is smooth and bounded since $u_1^*\in C^\infty(\Lambda_{t_0})$. Furthermore, recalling $e^{-2s}(\partial_Tu_T^*\circ\eta_T)|_{T=1}(s,y)=432e^sf_{1,1}^*(y)$ along with a similar claim for the $s$ derivative yields the first term in the stated expansion. For the second term, we can write
					$$
						u_\text{lin}(1+\beta+e^{-s}h(y),e^{-s}y)-u_\text{lin}(1+e^{-s}h(y),e^{-s}y)=r_2(\beta,s,y)\beta
					$$
				where
					$$
						r_2(\beta,s,y)=\int_0^1\partial_\beta u_\text{lin}(1+\beta x+e^{-s}h(y),e^{-s}y)dx
					$$
				which is smooth and bounded since  $u_\text{lin}\in C^\infty(\Lambda_{t_0})$. Recalling 
					$$
						e^{-2s}\begin{pmatrix}
							u_\text{lin}\circ\eta_{1}
							\\
							\partial_s(u_\text{lin}\circ\eta_{1})
						\end{pmatrix}\Bigg|_{s=s_0}
						=\chi_{t_0}\mathbf f_4^*
					$$
				yields the second claimed term in the expansion. $O(\alpha\beta)$ and $O(\beta^2)$ terms are obtained in $\mathbf V$ from $r_1$ and $r_2$ along with their $s$ derivatives. Lastly, the $O(\frac{\delta}{M_0^2})$ term in $\mathbf V$ follows from Inequality \eqref{uniform control}.
			\end{proof}
		
		\subsection{Hyperboloidal Evolution}
			At this point, we are ready to continue the evolution of the data we began to evolve in Section \ref{Evolution of Physical Initial Data}. We achieve this by evolving the data $\mathbf U\big((f,g),\alpha,\beta\big)$ according to the nonlinear theory developed in Proposition \ref{mod wp}. By a fixed point argument, we will show that there is at least one choice of $\alpha$ and $\beta$ for which the correction term vanishes. In other words, there is at least one choice of $\alpha$ and $\beta$ for which evolving $\mathbf U\big((f,g),\alpha,\beta\big)$ according to Equation \eqref{modified Duhamel} is equivalent to that of the quadratic wave equation.
			
			\begin{proposition} \label{hyperboloidal evolution}
				For all sufficiently large $M_0>0$, there exists $\delta>0$ such that for any pair $(f,g)\in\mathcal B_{\delta,M_0}$, there exists $(\alpha_{f,g},\beta_{f,g})\in[-\frac{\delta}{M_0},\frac{\delta}{M_0}]^2$ and a unique $\Phi_{f,g}\in C([s_0,\infty),\mathcal H_R)$ that satisfies
					$$
						\Phi_{f,g}(s)=\mathbf S(s-s_0)\mathbf U((f,g),\alpha_{f,g},\beta_{f,g})+\int_{s_0}^s\mathbf S(s-s')\mathbf N(\Phi_{f,g}(s'))ds', \label{Phifg eqn}
					$$
				and $\|\Phi_{f,g}(s)\|_{\mathcal H_R}\leq\delta e^{-\omega_0s}$ for all $s\geq s_0$.
			\end{proposition}
			\begin{proof}
				First, let $\delta>0$ be sufficiently small and $M_0>0$ sufficiently large so that the initial data operator is well-defined. Furthermore, the expansion of the initial data operator, Equation \eqref{initial data operator expansion}, implies that $\big\|\mathbf U\big((f,g),\alpha,\beta)\big)\big\|_{\mathcal H_R}\lesssim\frac{\delta}{M_0}$ for all $(f,g)\in\mathcal B_{\delta,M_0}$ and $(\alpha,\beta)\in[-\frac{\delta}{M_0},\frac{\delta}{M_0}]^2$. Thus, we require $M_0\gtrsim c$ so that $\big\|\mathbf U\big((f,g),\alpha,\beta)\big)\big\|_{\mathcal H_R}\leq\frac{\delta}{c}$ for $\delta$ and $c$ as in Proposition \ref{mod wp}. Thus, for $(f,g)\in\mathcal B_{\delta,M_0}$ and $(\alpha,\beta)\in[-\frac{\delta}{M_0},\frac{\delta}{M_0}]^2$, Proposition \ref{mod wp} implies the existence of a unique $\Phi_{f,g,\alpha,\beta}\in C([s_0,\infty),\mathcal H_R)$ that satisfies
					\begin{align*}
						\Phi_{f,g,\alpha,\beta}(s)=&\mathbf S(s-s_0)\bigg[\mathbf U\big((f,g),\alpha,\beta\big)-\mathbf C\Big(\Phi_{f,g,\alpha,\beta},\mathbf U\big((f,g),\alpha,\beta\big)\Big)\bigg]
						\\
						&+\int_{s_0}^s\mathbf S(s-s')\mathbf N(\Phi_{f,g,\alpha,\beta}(s'))ds'
					\end{align*}
				with the stated decay. If $\mathbf C\Big(\Phi_{f,g,\alpha,\beta},\mathbf U\big((f,g),\alpha,\beta\big)\Big)=\mathbf0$, then we are done. To this end, define the function $\Gamma_{f,g}:[-\frac{\delta}{M_0},\frac{\delta}{M_0}]^2\to\mathbb R^2$ by $\Gamma_{f,g}=\Big(\Gamma^{(1)}_{f,g},\Gamma^{(4)}_{f,g}\Big)$ where
					$$
						\Gamma^{(1)}_{f,g}(\alpha,\beta)=\bigg(\mathbf C_1\Big(\Phi_{f,g,\alpha,\beta},\mathbf U\big((f,g),\alpha,\beta\big)\Big)\bigg|\mathbf f_1^*\bigg)_{\mathcal H_R}
					$$
					$$
						\Gamma^{(4)}_{f,g}(\alpha,\beta)=\bigg(\mathbf C_4\Big(\Phi_{f,g,\alpha,\beta},\mathbf U\big((f,g),\alpha,\beta\big)\Big)\bigg|\mathbf f_4^*\bigg)_{\mathcal H_R}
					$$
				Proposition \ref{mod wp} and continuity of the initial data operator implies continuity of $\Gamma_{f,g}$. According to the expansion of the initial data operator and transversality of the Riesz projections, there exists a nonzero constant $\tilde\gamma_{t_0}$ such that
					$$
						\Gamma^{(1)}_{f,g}(\alpha,\beta)=\tilde\gamma_{t_0}\beta+\alpha(\mathbf P_1(\chi_{t_0}\mathbf f_4^*)|\mathbf f_1^*)_{\mathcal H_R}+\phi^{(1)}_{f,g}(\alpha,\beta)
					$$
				and
					$$
						\Gamma^{(4)}_{f,g}(\alpha,\beta)=\alpha(\mathbf P_4(\chi_{t_0}\mathbf f_4^*)|\mathbf f_4^*)_{\mathcal H_R}+\phi^{(4)}_{f,g}(\alpha,\beta)
					$$
				where $\phi_{f,g}=\Big(\phi^{(1)}_{f,g}(\alpha,\beta),\phi^{(4)}_{f,g}(\alpha,\beta)\Big):[-\frac{\delta}{M_0},\frac{\delta}{M_0}]^2\to\mathbb R^2$ is continuous and satisfies the estimate $|\phi_{f,g}|\lesssim\frac{\delta}{M_0^2}+\delta^2$.
			
				The equation $\Gamma_{f,g}(\alpha,\beta)=0$ is equivalent to the existence of a fixed point of the map $(\alpha,\beta)\mapsto\mathbf A_{t_0}^{-1}\phi_{f,g}(\alpha,\beta)$ where 
					$$
						\mathbf A_{t_0}=-\begin{pmatrix}(\mathbf P_1(\chi_{t_0}\mathbf f_4^*)|\mathbf f_1^*)_{\mathcal H_R}&\tilde\gamma_{t_0}\\(\mathbf P_4(\chi_{t_0}\mathbf f_4^*)|\mathbf f_4^*)_{\mathcal H_R}&0\end{pmatrix}.
					$$ 
				This matrix is invertible since $\tilde\gamma_{t_0},(\mathbf P_4(\chi_{t_0}\mathbf f_4^*)|\mathbf f_4^*)_{\mathcal H_R}\neq0$, the second following from Lemma \ref{nonvanishing of the Riesz projection}. Denoting by $\|\cdot\|_{M_2(\mathbb C)}$ the matrix norm on $M_2(\mathbb C)$, we have that
					\begin{align*}
						|\mathbf A_{t_0}^{-1}\phi_{f,g}(\alpha,\beta)|&\leq\|\mathbf A_{t_0}^{-1}\|_{M_2(\mathbb C)}|\phi_{f,g}(\alpha,\beta)|
						\\
						&\lesssim\Big(\frac{\delta}{M_0^2}+\delta^2\Big).
					\end{align*}
				Thus, for $M_0$ sufficiently large, one can take $\delta\lesssim\frac{1}{2M_0}$ in order to show that this map sends $[-\frac{\delta}{M_0},\frac{\delta}{M_0}]^2$ to itself. Thus, the Brouwer fixed point theorem implies the existence of a fixed point.
			\end{proof}
		
		\subsection{Proof of the Main Result}
			Fix $M_0>0$ sufficiently large and let $\delta>0$ be as in Proposition \ref{hyperboloidal evolution}. By Proposition \ref{hyperboloidal evolution}, we have that for all $(f,g)\in\mathcal B_{\delta,M_0}$, there exist $(\alpha,\beta)\in[-\frac{\delta}{M_0},\frac{\delta}{M_0}]^2$ and a unique $\Phi=(\phi_1,\phi_2)\in C([s_0,\infty),\mathcal H_R)$ solving Equation \eqref{Phifg eqn}. Observe that 
				$$
					\mathbf U((f,g),\alpha,\beta):=e^{-2s}
					\begin{pmatrix}
						u_{f,g,\alpha}\circ\eta_{1+\beta}-u_{1+\beta}^*\circ\eta_{1+\beta}
						\\
						\partial_s(u_{f,g,\alpha}\circ\eta_{1+\beta})-\partial_s(u_{1+\beta}^*\circ\eta_{1+\beta})
					\end{pmatrix}\Bigg|_{s=s_0}
				$$
			belongs to $C^\infty(\mathbb B_R^7)^2$. Thus, an inductive argument analogous to the proof of Lemma \ref{adjustment term} shows that $\Phi$ is indeed smooth and a classical solution of Equation \eqref{NL eqn}. By setting $T=1+\beta$, we see that the function $\tilde u$ defined by the equation
				$$
					(\tilde u\circ\eta_T)(s,y)=(u_T^*\circ\eta_T)(s,y)+e^{2s}\phi_1(s)(y)
				$$
			solves Equation \eqref{qwe} on $\eta_T([s_0,\infty)\times[0,R))$ and further satisfies
				$$
					\partial_s(\tilde u\circ\eta_T)(s,y)=\partial_s(u_T^*\circ\eta_T)(s,y)+e^{2s}\phi_2(s)(y).
				$$
			Setting $u(t,x)=\tilde u(t,|x|)$, we infer by uniqueness in Lemma \ref{local adjusted CP} and finite speed of propagation, Theorem 2.12 of \cite{BDS19} for instance, that $u\in C^\infty(\Omega_{T,R})$ and solves the Cauchy problem \eqref{qwe adjusted CP}. Lastly, the stated convergence follows from the decay of $\Phi_{f,g}$.
	
	\appendix	
		\section{Explicit Expressions for Proposition \ref{spectrum}} \label{explicit}
			$C_n(\lambda)=\frac{P_1(n,\lambda)}{P_2(n,\lambda)}$ and $\varepsilon_n(\lambda)=\frac{P_3(n,\lambda)}{P_2(n,\lambda)}$ where
				$$
					P_1(n,\lambda)=-845000000 n^2 (n+1)^3 (2 n+11) (\lambda +2 n+2) (\lambda +2 n+8),
				$$
				\begin{align*}
					P_2(n,\lambda)=&\big(1287\lambda ^2+52 (192 \lambda ^2+4125 \lambda +9500) n^2+2000 (48 \lambda+299) n^3+104000 n^4
					\\
					&+\lambda  (1521 \lambda -44000) n\big)\big(52 (246\lambda ^2+5125 \lambda +23000)
					\\
					&+4 (2496 \lambda ^2+125625 \lambda+728000) n^2
					\\
					&+6000 (16 \lambda +169) n^3+104000 n^4+(21489 \lambda ^2+673000\lambda +3198000) n\big),
				\end{align*}
			and
				\begin{align*}
					P_3(n,\lambda)=&-33462 \lambda ^2 (117 \lambda ^2-4000 \lambda -4000)-4000 (74344 \lambda^2+1012375 \lambda +14196000) n^6
					\\
					&+1000 (57408 \lambda ^3-1058203 \lambda^2-34820500 \lambda -237276000) n^5
					\\
					&-8 (292032 \lambda ^4-66226875 \lambda^3-854777500 \lambda ^2+11226312500 \lambda +52728000000) n^4
					\\
					&-10 (2021409\lambda ^4-113476350 \lambda ^3-1831007400 \lambda ^2
					\\
					&\;\;\;\;\;\;+9749350000 \lambda+33360600000) n^3
					\\
					&-5 (6739551 \lambda ^4-166553400 \lambda ^3-1369066800\lambda ^2
					\\
					&\;\;\;\;\;\;+8543600000 \lambda +19468800000) n^2
					\\
					&+78000000 (4 \lambda -65) n^7-325\lambda  (22113 \lambda ^3-1579680 \lambda ^2
					\\
					&\;\;\;\;\;\;+11482600 \lambda +14080000) n.
				\end{align*}
			Also, $\delta_5(\lambda)=\frac{R_1(\lambda)}{R_2(\lambda)}$ where
				\begin{align*}
					R_1(\lambda)=&-597051 \lambda ^{12}-43222410 \lambda ^{11}+5068245600 \lambda ^{10}+633420595440\lambda ^9
					\\
					&+23910688879632 \lambda ^8+308544639036000 \lambda ^7-3181221429731200\lambda ^6
					\\
					&-155692128689456640 \lambda ^5-2167560072357216256 \lambda^4-15251720333529661440 \lambda ^3
					\\
					&-55976373542617907200 \lambda^2-95372978774016000000 \lambda -51994908426240000000
				\end{align*}
			and
				\begin{align*}
					R_2(\lambda)=&2 (64623 \lambda ^2+4285625 \lambda +38025000) 
					\\
					&\times\big(81 \lambda ^{10}+19710\lambda ^9+1886400 \lambda ^8+92781360 \lambda ^7+2577603408 \lambda ^6+41940364000\lambda ^5
					\\
					&+401332867200 \lambda ^4+2206815715840 \lambda ^3+6537890727936 \lambda^2+8994221424640 \lambda 
					\\
					&+3845731123200\big).
				\end{align*}	
	\section{The wave propagators in light cones }\label{Cauchy_Theory}

Recall that the solution of the Cauchy problem
				$$
					\begin{cases}
						\big(\partial_t^2-\Delta_x\big)u(t,x)=0&(t,x)\in\mathbb R^{1+d}
						\\
						u[0](x)=\big(f(x),g(x)\big)&x\in\mathbb R^d
					\end{cases}
				$$
			is given by the wave propagators
				$$
					u(t,\cdot)=\cos(t|\nabla|)f+\frac{\sin(t|\nabla|)}{|\nabla|}g
				$$
			for $f,g\in\mathcal S(\mathbb R^d)$ where $\phi(|\nabla|)f=\mathcal F^{-1}(\phi(|\cdot|)\mathcal Ff)$ for $\phi\in C(\mathbb R)$. In backward light cones, they satisfy the following estimates.
				\begin{proposition} \label{finite speed of propagation}
					Let $d\geq3$. Then there exists a continuous function $\gamma_d:[0,\infty)\to[1,\infty)$ such that
	\begin{align*}
							\big\|\partial_t^\ell\cos(t|\nabla|)f\big\|_{\dot{H}^k(\mathbb B_{T-t}^d)} & \leq\|f\|_{\dot{H}^{k+\ell}(\mathbb B_{T}^d)} \\
								\bigg\|\partial_t^\ell\frac{\sin(t|\nabla|)}{|\nabla|}f\bigg\|_{\dot{H}^k(\mathbb B_{T-t}^d)} & \leq\|f\|_{\dot{H}^{k+\ell-1}(\mathbb B_{T}^d)},
	\end{align*}
as well as 					
	\begin{align*}
							\big\|\partial_t^\ell\cos(t|\nabla|)f\big\|_{L^2(\mathbb B_{T-t}^d)} & \leq\gamma_d(T)\|f\|_{\dot{H}^{1+\ell}(\mathbb B_{T}^d)}, \\
							\bigg\|\partial_t^\ell\frac{\sin(t|\nabla|)}{|\nabla|}f\bigg\|_{L^2(\mathbb B_{T-t}^d)} & \leq\gamma_d(T)\|f\|_{\dot{H}^{\ell}(\mathbb B_{T}^d)},
	\end{align*}
					for all $f\in\mathcal S(\mathbb R^d)$, $T>0$, $t\in[0,T)$, $k\in\mathbb N$, and $\ell\in\mathbb N_0$.
				\end{proposition}
For the proof, we refer the reader to \cite{BDS19}. By density, these estimates hold for data in homogeneous Sobolev spaces as well. 
\bibliographystyle{plain}
\bibliography{qwe_bibfile}

\end{document}